\documentclass[11pt,twoside, leqno]{article}

\usepackage{amsthm}
\usepackage{amssymb}
\usepackage{amsmath}
\usepackage{mathrsfs}
\usepackage{txfonts}
\usepackage[active]{srcltx}

\allowdisplaybreaks

\def\rr{{\mathbb R}}
\def\rn{{{\rr}^n}}

\def\zz{{\mathbb Z}}

\def\cn{{\mathbb N}}

\def\cm{{\mathcal M}}
\def\car{{\mathcal R}}

\def\fz{\infty}
\def\az{\alpha}
\def\supp{{\mathop\mathrm{\,supp\,}}}

\def\loc{{\mathop\mathrm{\,loc\,}}}
\def\Lip{{\mathop\mathrm{\,Lip\,}}}
\def\lip{{\mathop\mathrm{\,lip\,}}}
\def\lz{\lambda}
\def\dz{\delta}

\def\kz{\kappa}
\def\bz{\beta}
\def\ro{\rho}

\def\gz{{\gamma}}

\def\sz{\sigma}
\def\pa{\partial}

\def\wz{\widetilde}
\def\nz{\nabla}

\def\hs{\hspace{0.3cm}}

\def\r{\right}
\def\lf{\left}

\def\la{\langle}
\def\ra{\rangle}

\newtheorem{thm}{Theorem}[section]
\newtheorem{lem}{Lemma}[section]
\newtheorem{prop}{Proposition}[section]
\newtheorem{rem}{Remark}[section]
\newtheorem{cor}{Corollary}[section]

\numberwithin{equation}{section}

\begin{document}
\arraycolsep=1pt
\title{\Large\bf Gradient Estimate for Solutions to Poisson Equations
\\in Metric Measure Spaces\footnotetext{\hspace{-0.35cm} 2000 {\it
Mathematics Subject Classification}. 31C25; 31B05; 35B05; 35B45
\endgraf{\it Key words and phrases. {\rm Moser-Trudinger inequality; Poincar\'e inequality;
Poisson equation; Riesz Potential; Sobolev inequality; curvature
\endgraf Renjin Jiang was partially supported by the Academy of Finland grants 120972 and
131477.}}
}}

\author{Renjin Jiang}
\date{ }
\maketitle
\begin{center}
\begin{minipage}{13.5cm}\small
{\noindent{\bf Abstract.} Let $(X,d)$ be a complete, pathwise
connected metric measure space with a locally Ahlfors $Q$-regular measure $\mu$, where $Q>1$.
Suppose that $(X,d,\mu)$ supports a (local) $(1,2)$-Poincar\'e inequality
and a suitable curvature lower bound. For the Poisson equation $\Delta u=f$ on
$(X,d,\mu)$, Moser-Trudinger and Sobolev inequalities are established
for the gradient of $u$. The local H\"older
continuity with optimal exponent of solutions is obtained. }
\end{minipage}
\end{center}
\vspace{0.0cm}

\section{Introduction}
\hskip\parindent  Let $M$ be an $n$-dimensional ($n\ge 2$) complete, connected
Riemannian manifold with Riemannian metric $\ro$. Denote by $\Delta$,
$\nz$ the Laplace-Beltrami operator and the gradient on $M$, respectively.
Assume that the Ricci curvature is bounded from below by a constant $K\in \rr$,
i.e.,
\begin{equation}\label{1.1x}
Ric_x(X,X)\ge -K|X|^2,\ \ \ \forall \, x\in M, \ X\in T_xM.
\end{equation}

Let $p$ and $\{P_t\}_{t>0}$ be the heat kernel and heat semigroup of the
Laplace-Beltrami operator on $M$, respectively.
In 1986, a breakthrough was made by Li and Yau in \cite{ly86}, where they obtained
pointwise estimates on $p$ and the gradient of $p$, $\nz p$. When $M$
has non-negative Ricci-curvature, their estimates read as:
$$\frac{C}{V(x,\sqrt t)}\exp\lf\{-\frac{\ro(x,y)^2}{ct}\r\}\le p(x,y,t)\le
\frac{C}{V(x,\sqrt t)}\exp\lf\{-\frac{\ro(x,y)^2}{\wz c t}\r\},$$
$$|\nz_x p(x,y,t)|\le \frac{C}{\sqrt t V(x,\sqrt t)}\exp\lf\{-\frac{\ro(x,y)^2}{ct}\r\},$$
where $V(x,\sqrt t)$ denotes the volume of the metric ball $B(x,\sqrt t)$. Li-Yau type
estimates have turned out to be powerful tools in many branches of modern mathematics,
see, for example, \cite{lst01,wz03} for applications to Poisson equation
on Riemannian manifold with non-negative Ricci curvature.

On the other hand, Gross \cite{gro1} derived the remarkable Gaussian Sobolev
inequality
$$\int_\rn |f(x)|^2\ln |f(x)|\,d\nu(x)\le \int_{\rn}|\nz f(x)|^2\,d\nu(x)+
\|f\|^2_{L^2(\nu)}\ln \|f\|_{L^2(\nu)},$$
where $\nu$ denotes the Gaussian measure on $\rn$, which is also referred to
as the logarithmic Sobolev inequality. While the classical Sobolev inequality highly
depends on the dimension $n$, the logarithmic Sobolev inequality is uniform
in all dimension $n$, which enables one to extend it to infinite dimension.
Moreover, when passing from Euclidean spaces to Riemannian manifolds, the
logarithmic Sobolev inequality (in different forms) even reflects some
deep geometric properties.

Recall that ``square of the length of the gradient", which is due to Bakry and
Emery \cite{be2}, is defined as
$$\Gamma_2(u,u)=\frac 12\Delta(|\nz u|^2)-\la\nz\Delta u,\nz u\ra,\ \ u\in C^\fz(M).$$
The diffusion semigroup is said to have curvature greater or equal to some
$K\in\rr$, if
\begin{equation}\label{1.2x}
\Gamma_2(u,u)\ge -K\la \nz u,\nz u\ra, \ \ \forall u\in C^\fz(M).
\end{equation}
It is well known that \eqref{1.2x} is equivalent to $\eqref{1.1x}.$
Moreover, they are all equivalent to:
\begin{equation}\label{1.3x}
P_t(u^2)-(P_t u)^2\le \frac{e^{2Kt}-1}{K}P_t(|\nz u|^2), \ \ \forall t\ge 0, \ \forall u\in C_c^\fz(M),
\end{equation}
\begin{equation*}
P_t(u^2\log u^2)-(P_t u^2)\ln (P_t u^2)\le \frac{2(e^{2Kt}-1)}
{K}P_t(|\nz u|^2), \ \ \forall t\ge 0, \ \forall u\in C_c^\fz(M),
\end{equation*}
see \cite{bak1}. Wang \cite{wa04} showed that \eqref{1.1x} is also equivalent to
the so-called dimension-free Harnack inequality; see also \cite{wa97}.

Our main aim in this paper is to provide a semigroup approach via
the logarithmic Sobolev inequality \eqref{1.3x}, instead of Li-Yau type estimates
for the gradient of the heat kernel, to study the
local behavior of solutions to the Poisson equation $\Delta u=f$. Taking a Riemannian manifold that
satisfies \eqref{1.3x} as a guiding example, we will single out the crucial
assumptions necessary for our semigroup approach, by formulating the arguments in an
abstract metric space. Our results indicate that already the logarithmic Sobolev inequality
\eqref{1.3x} together with a $2$-Poincar\'e inequality (see \eqref{1.1} below) is sufficient
to guarantee Euclidean type local behavior of solutions to Poisson equation.

Let us now describe the metric setting. Let $(X,d)$ be a complete, pathwise connected
metric measure space. Suppose that  $(X,d)$ is endowed with a
locally $Q$-regular measure $\mu$, $Q>1$, where local $Q$-regularity means
that there exist constants $C_Q\ge1$ and $R_0\in (0,\fz]$
such that for every $x\in X$ and
all $r\in (0, R_0)$,
$$C_Q^{-1}r^Q\le \mu(B(x,r))\le C_Qr^Q.$$
The reader interested in Riemannian manifolds should here think $X$ to be a
weighted Riemannian manifold.

By the work of Buser \cite{bu82}, each complete Riemannian manifold with Ricci-curvature bounded from below
admits a local $2$-Poincar\'e inequality. Correspondingly, we assume a (weak) $2$-Poincar\'e inequality
on $(X,d,\mu)$. That is, there exist $C_P>0$ and $\lz\ge 1$ such that for all Lipschitz
functions $u$ and each ball $B_r(x)=B(x,r)$ with $r< R_0$,
\begin{equation}\label{1.1}
\fint_{B_r(x)}|u-u_{B_r(x)}|\,d\mu\le
C_Pr\lf(\fint_{B_{\lz r}(x)}[\Lip u]^2\,d\mu\r)^{1/2},
\end{equation}
where and in what follows, for each ball $B\subset X$, $u_B=\fint_B u\,d\mu=\mu(B)^{-1}\int_{B}u\,d\mu$, and
\begin{equation*}
\Lip u(x)=\limsup_{r\to 0}\sup_{d(x,y)\le r}\frac{|u(x)-u(y)|}{r}.
\end{equation*}
Although our results work for $\lz>1$ as well, we will assume throughout the
paper, that $\lz=1$, for simplicity. See \cite{kei03,hek,kz08} for more about the
Poincar\'e inequality on metric measure spaces.

%From \cite{kei03}, the Poincar\'e inequality here coincides with the one
%for all measurable functions and their upper gradients. Recall that for a domain $\Omega \subseteq X$
%and a measurable function $u$ on $\Omega$, a non-negative Borel
%function $g$ is called an upper gradient of $u$ on $\Omega$, if
%\begin{equation*}
 %|u(x)-u(y)|\le \int_{\gz} g\,ds
%\end{equation*}
%for all $x,\,y\in\Omega$ and each rectifiable curve $\gz:\,[0,l]\to
%\Omega$ that joins $x$ and $y$; see \cite{hek}. Then, a metric measure space $(X,d,\mu)$ is
%said to support a (weak) $p$-Poincar\'e inequality, if there exist
%$C_P>0$ and $\lz\ge1$ such that for each ball $B(x,r)\subseteq X$ and
%for all continuous functions $u$ and all upper gradients $g$ of $u$
%on $B(x,\lz r)$,
%\begin{equation*}
%\fint_{B_r(x)}|u-u_B|\,d\mu\le
%C_Pr\lf(\fint_{B_{\lz r}(x)}g^p\,d\mu\r)^{1/p}.
%\end{equation*}
%According to \cite{kz08}, the space $(X,d,\mu)$ actually supports
%a (local) $(2-\ez)$-Poincar\'e inequality for some $\ez>0$.

For a locally Lipschitz continuous function $u$, define its $H^{1,p}(X)$ norm ($p>1$) by
$$\|u\|_{H^{1,p}(X)}:=\|u\|_{L^p(X)}+\|\Lip u\|_{L^p(X)}.$$
Then the Sobolev space $H^{1,p}(X)$ is defined to be the completion
of the set of all locally Lipschitz continuous functions $u$ with $\|u\|_{H^{1,p}(X)}<\fz$.
By the work of Cheeger \cite{ch}, we can assign a derivative to each Lipschitz
function $u$. In what follows, let $D$ be a Cheeger derivative operator
in $(X,d,\mu)$. It is shown in \cite{ch} that $|Du|$ is comparable to $\Lip u$ for each locally Lipschitz
continuous function $u$, and $D$ satisfies the Leibniz rule;
see Section 2 for details. Actually,  the construction of $D$ is irrelevant
for our approach as long as $D$ has the properties above and comes with an
associated inner product, with $Du\cdot Du$ comparable to the square of $\Lip u$.
In the Riemannian setting, we simply consider $\nz u$ with the Riemannian inner
product $\la \nz u,\nz \phi\ra$.
The local Sobolev space $H^{1,p}_\loc(X)$ is defined
as usual. For an open set $U\subset X$, the space $H^{1,p}_0(U)$
is defined to be the closure in $H^{1,p}(X)$ of Lipschitz functions with compact support in $U$.

Let $\Omega\subseteq X$ be a domain.
As in the Riemannian setting, a Sobolev function $u\in H^{1,2}(\Omega)$
is called a solution of $\Delta u=g$ in $\Omega$, if
\begin{equation}\label{1.2}
-\int_\Omega Du(x)\cdot D\phi(x)\,d\mu(x)=\int_\Omega g(x)\phi(x)\,d\mu(x),
\ \ \forall \phi\in H_0^{1,2}(\Omega).
\end{equation}
Biroli and Mosco \cite{bm} studied the Poisson equation by assuming that
$\mu$ is doubling and that a $2$-Poincar\'e inequality holds. In their paper,
the Green function, existence of solutions and H\"older
continuity of solutions are studied. We remark that the H\"older continuity
in \cite{bm} is obtained from Moser iteration and the exponent of H\"older continuity
is not of exact form. For potential theory on metric spaces, we refer to \cite{bb10}.

Our main aim is to establish a Moser-Trudinger type
inequality and Sobolev inequality for the gradients of solutions. Thus, modelling \eqref{1.3x},
we assume the following curvature condition.
Assume that there exists  a nonnegative function $c_{\kz}(T)$ on $(0,\fz)$ such that
for each $0<t<T$ and every $g\in H^{1,2}(X)$, we have
\begin{eqnarray}\label{1.3}
\int_X g(y)^2 p(t,x,y)\,d\mu(y)&&\le (2t+c_{\kz}(T)t^2)\int_X |Dg(y)|^2 p(t,x,y)\,d\mu(y)\nonumber\\
&&\hs+\lf(\int_X g(y)p(t,x,y)\,d\mu(y)\r)^2
\end{eqnarray}
for almost every $x\in X$, where $p(t,x,y)$ refers to the
heat kernel associated to the Dirichlet form $\int_X Df\cdot Dg\,d\mu$,
see Section 2 for details. In the Riemannian setting, $p$ is the usual heat kernel.
The function $c_{\kz}(T)$ should be viewed as
a consequence of some abstract lower curvature bound $-\kz$, and it is non-decreasing
as one can deduce from the assumption. Many examples in the classical smooth setting can be
found in \cite{bak1,be2,chy,gro1,wa04,wa97}.

Further examples include compact Alexandrov spaces with curvature bounded from below.
It is well known that the (local) Poincar\'e inequality \eqref{1.1} holds on
Alexandrov spaces with curvature bounded from below; see, for instance,
\cite{zz10}. Very recently, Gigli et al verified that \eqref{1.3} holds on
them, see \cite[Theorem 4.3]{gko11}.

Lott and Villani (\cite{lv09}) and Sturm (\cite{stm4,stm5}) independently
introduced and analyzed Ricci curvature in metric measure spaces via optimal mass transportation.
On a metric space with Ricci curvature (in the sense of Lott-Sturm-Villani) bounded from below
that additionally satisfies a local angle condition, a semi-concavity condition and
that the pointwise Lipschitz constant coincides with the length of the gradient,
\eqref{1.3} holds by results of Koskela and Zhou \cite[Corollary 6.2]{kz11}
(that employ the contraction property of the
gradient flow of entropy due to Savar\'e \cite{sa07}).

Koskela et al \cite{krs} established the Lipschitz regularity of Cheeger-harmonic (i.e.\,$\Delta u=0$)
functions under the above assumptions. They also showed for the space
$(X_\az,|\cdot|,dx)$, where $|\cdot|$ denotes the Euclidean metric, $dx$ the Lebesgue measure,
$\az\in (\pi,2\pi)$,
$$X_\az=\{(r\cos \phi,r\sin \phi)\in \rr^2:\,\phi\in [0,\az],r\ge 0\},$$
that \eqref{1.3} does not hold and that there exists a Cheeger-harmonic function
which is not locally Lipschitz continuous. On the other hand,
the space $(X_\az,|\cdot|,dx)$ with $\az\in (0,\pi]$ satisfies our assumptions.
Under the same assumptions,
for the Poisson equation $\Delta u=g$, the local Lipschitz continuity
of solutions $u$ is established when $g\in L^p$ with $p>Q$ in \cite{ji}.

We are in position to state our first gradient estimate.
\begin{thm}\label{t1.1}
Let $Q\in (1,\fz)$ and assume that \eqref{1.1} and \eqref{1.3} hold.
Then there exist $c,C>0$ such that
for all $u\in H^{1,2}(8B)$ and $g\in L^Q(8B)$ that satisfy
$\Delta u=g$ in $8B$, where $B=B_R(y_0)$ with $256R< R_0$,
\begin{equation*}
\fint_B \exp\lf(\frac{c|Du(x)|}{(1+\sqrt {c_{\kz}(R^2)}R)C(u,g)}\r)^{\frac {Q}{Q-1}}\,d\mu(x)\le C,
\end{equation*}
where $C(u,g)=R^{-Q/2-1}\|u\|_{L^2(8B)}+\|g\|_{L^Q(8B)}$.
\end{thm}

The technical requirement $8B$ and $R<R_0/256$ can certainly be relaxed.
The point is that, in the abstract setting,
when dealing with an equation that $\Delta u=g$ in $\lz B$ for some $\lz>1$,
we need to consider an auxiliary equation in a ball bigger than $\lz B$; see our arguments in Section 4.

Let us consider the Poisson equation $\Delta u=g$ with
$g\in L^p_\loc(X)$ and $p<Q$. Since $u$ belongs to $H^{1,2}_\loc(X)$ by definition,
it is then natural to restrict
$p\in (2_\ast,Q)\cap (1,Q)$, where $2_\ast=\frac{2Q}{Q+2}$. Notice that $2_\ast<1$ only
for $Q<2$. We have the following result.
\begin{thm}\label{t1.2}
Let $Q\in(1,\fz)$, $p\in (2_\ast,Q)\cap (1,Q)$ and assume that \eqref{1.1} and \eqref{1.3} hold.
Then there exists a constant $C$ such that for all $u\in H^{1,2}(8B)$
and $g\in L^p(8B)$ that satisfy $\Delta u=g$ in $8B$,
$$\lf(\fint_B |Du|^{p^\ast}\,d\mu\r)^{1/p^\ast}\le
C(1+\sqrt{c_{\kz}(R^2)}R)\lf\{R^{-1}\lf(\fint_{8B}|u|^2\,d\mu\r)^{1/2}+R\lf(\fint_{8B}|g|^p\,d\mu\r)^{1/p}\r\},$$
 where $B=B_R(y_0)$ with $R<R_0/256$ and $p^\ast=\frac{Qp}{Q-p}$.
\end{thm}

How to prove the above results? As mentioned above, we use a semigroup
approach. This method was introduced in \cite{ck} in the Euclidean setting
 to study variable coefficient parabolic equations,
and was applied in \cite{krs} to Lipschitz continuity of Cheeger-harmonic functions;
see Section 3 below. By using this method, for the auxiliary equation $\Delta v=g\chi_{8B}$ in $256B$,
we obtain a pointwise estimate for the gradient of $v$ by generalized Riesz potentials
based on the heat semigroup.
By using the mapping properties of the generalized Riesz potentials, we then establish the
above two theorems for the solutions of the auxiliary equations. Then, for general solutions
of the Poisson equation, Theorem \ref{t1.1} and Theorem \ref{t1.2} follow by using
density arguments and the theory of Cheeger-harmonic functions.

As a corollary to Theorem \ref{t1.2}, we have the following H\"older-continuity
estimate.

\begin{cor}\label{c1.1}
Let $Q\in(1,\fz)$, $p\in (\frac Q2,Q)\cap (1,Q)$ and assume that \eqref{1.1} and \eqref{1.3} hold.
Suppose that $u\in H^{1,2}_{\loc}(\Omega)$ satisfies $\Delta u=g$ with $g\in L^p_\loc(\Omega)$,
where $\Omega\subseteq X$ is a domain.
Then $u$ is locally H\"older continuous with exponent $2-\frac Qp$ in $\Omega$.
\end{cor}

The paper is organized as follows. In Section 2, we give some basic notation
and notions for Cheeger derivatives, Dirichlet forms and
Orlicz spaces. Several auxiliary results regarding Poisson equations are also given
in Section 2. Section 3 is devoted to introducing the method and some estimates.
We study auxiliary equations in Section 4 and prove Theorem \ref{t1.1} and Theorem
\ref{t1.2} for the solutions of the auxiliary equations.
The main results are proved in Section 5.

Finally, we make some conventions. Throughout the paper, we denote
by $C,c$ positive constants which are independent of the main
parameters, but which may vary from line to line. The symbol
$B_R(x)=B(x,R)$ denotes an open ball
with center $x$ and radius $R$ and $B_{CR}(x)=CB_R(x)=B(x,CR).$
For $p\in (1,Q)$, denote $\frac{Qp}{Q-p}$ by $p^\ast$, and
for $p\in(1,\fz)$, denote $\frac{Qp}{Q+p}$ by $p_\ast$.

\section{Preliminaries}
\hskip\parindent In this section, we give some basic notation
and notions and several auxiliary results.

\subsection{Cheeger Derivative in metric measure spaces}
\hskip\parindent Let $(X,d,\mu)$ be a metric measure space
with $\mu$ Ahlfors $Q$-regular for some $Q>1$.
Cheeger \cite{ch} generalized Rademacher's theorem of
differentiability of Lipschitz functions on $\rn$ to metric measure spaces.
Precisely, the following theorem provides us the differential structure.

\begin{thm}\label{t2.1}
Assume that $(X,\mu)$ supports a weak $p$-Poincar\'e inequality for some
$p>1$ and that $\mu$ is doubling. Then there exists $N > 0$, depending only on the doubling
constant and the constants in the Poincar\'e inequality, such that the following holds.
There exists a countable collection of measurable sets $U_\az$, $\mu(U_\az)>0$
for all $\az$, and Lipschitz functions $X^\az_1,\cdots, X^\az_{k(\az)}:\, U_\az\to \rr$,
with $1\le k(\az)\le N$ such that
$\mu\lf(X \setminus\cup_{\az=1}^\fz U_\az\r)=0$,
and for all $\az$ the following holds:
for $f: X\to \rr$ Lipschitz, there exist $V_\az(f)\subseteq U_\az$ such that
$\mu(U_\az\setminus V_\az(f))=0$, and Borel functions $b_1^\az(x,f),\cdots,b^\az_{k(\az)}(x,f)$
of class $L^\fz$ such that if $x\in V_\az(f)$, then
$$\Lip(f-a_1 X_1^\az-\cdots-a_{k(\az)}X^\az_{k(\az)})(x)=0$$
if and only if $(a_1,\cdots,a_{k(\az)})=(b^\az_1(x,f),\cdots,b^\az_{k(\az)}(x,f))$.
Moreover, for almost every $x\in U_{\az_1}\cap U_{\az_2}$, the ``coordinate functions"
$X_i^{\az_2}$ are linear combinations of the $X_i^{\az_1}$'s.
\end{thm}
By Theorem \ref{t2.1}, for each Lipschitz function $u$ we can assign a derivative $Du$,
which we call Cheeger derivative following \cite{krs}.
For each locally Lipschitz function $f$, we define $\lip f$ by
\begin{equation*}
\lip f(x)=\liminf_{r\to 0}\sup_{d(x,y)\le r}\frac{|f(x)-f(y)|}{r}.
\end{equation*}
By \cite{ch}, under the assumptions of Theorem \ref{t2.1},
for each locally Lipschitz $f$, $\Lip f$ and $\lip f$ coincide
with the minimal upper gradient $g_u$ of $u$ almost everywhere, and they
all are comparable to $|Du|$. See also \cite{ke04}.

By \cite{sh} and \cite{ch},
the Sobolev spaces $H^{1,p}(X)$ are isometrically equivalent to the Newtonian Sobolev spaces
$N^{1,p}(X)$ defined in \cite{sh} for $p\ge 2$. Franchi et al \cite{fhk} further showed
that the differential operator $D$ can be extended to all functions in the corresponding
Sobolev spaces. A useful fact is that the Cheeger derivative satisfies the Leibniz rule, i.e.,
for all $u,v\in H^{1,2}(X)$,
$$D(uv)(x)=u(x)Dv(x)+v(x)Du(x).$$

\subsection{Dirichlet forms and heat kernels}
\hskip\parindent Having defined the Sobolev spaces $H^{1,p}(X)$ and the differential operator $D$,
we now consider Dirichlet forms on $(X,\mu)$.
Define the bilinear form $\mathscr{E}$ by
$$\mathscr{E}(f,g)=\int_X Df(x)\cdot Dg(x)\,d\mu(x)$$
with the domain $D(\mathscr{E})=H^{1,2}(X)$.
It is easy to see that $\mathscr{E}$ is symmetric and closed.
Corresponding to such a form there exists
an infinitesimal generator $A$ which acts on a dense
subspace $D(A)$ of $H^{1,2}(X)$ so that for all $f\in D(A)$ and each
$g\in H^{1,2}(X)$,
$$\int_X g(x)Af(x)\,d\mu(x)=-\mathscr{E}(g,f).$$

Now let us recall several auxiliary results established in \cite{krs}.
\begin{lem}
If $u,\,v\in H^{1,2}(X)$, and $\phi\in H^{1,2}(X)$ is a bounded
Lipschitz function, then
\begin{equation*}
\mathscr{E}(\phi,uv)=\mathscr{E}(\phi u,v)+\mathscr{E}(\phi
v,u)-2\int_X \phi Du(x)\cdot Dv(x)\,d\mu(x).
\end{equation*}
Moreover, if $u,\,v\in D(A)$, then we can unambiguously
define the measure $A(uv)$ by setting
\begin{equation*}
A(uv)=uAv+vAu+2Du\cdot Dv.
\end{equation*}
\end{lem}

Also, associated with the Dirichlet form $\mathscr{E}$, there is a
semigroup $\{T_t\}_{t>0}$, acting on $L^2(X)$, with the following
properties (see \cite[Chapter 1]{fot}):

1. $T_t\circ T_s=T_{t+s},\,\forall \,t,\,s>0$,

2. $\int_X|T_t f(x)|^2\,d\mu(x)\le \int_X f(x)^2\,d\mu(x),\,\forall
\,f\in L^2(X,\mu) \ \mbox{and} \ \forall \, t>0$,

3. $T_t f\to f$ in $L^2(X,\mu)$ when $t\to 0$,

4. if $f\in L^2(X,\mu)$ satisfies $0\le f\le C$, then $0\le T_tf\le
C$ for all $t>0$,

5. if $f\in D(A)$, then $\frac 1t (T_tf-f)\to Af$ in
$L^2(X,\mu)$ as $t\to 0$, and

6. $A T_tf=\frac {\pa}{\pa_t}T_tf$, $\forall t>0$ and $\forall \,
f\in L^2(X,\mu)$.

A measurable function $p:\, \rr\times X\times X \to [0,\fz]$ is said
to be a heat kernel on $X$ if
$$ T_tf(x)=\int_X f(y)p(t,x,y)\,d\mu(y)$$ for
every $f\in L^2(X,\mu)$ and all $t\ge 0$, and $p(t,x,y)=0$ for every
$t<0$. Let the measure on $X$ be doubling (i.e. $\mu(2B)\le C_d\mu(B)$ for
each ball $B$) and assume that the
2-Poincar\'e inequality \eqref{1.1} holds. Sturm (\cite{st3}) proved the existence of a heat
kernel and a Gaussian estimate for the heat kernel, which in our settings reads
as: there exist positive constants $C,\,C_1,\,C_2$ such that
\begin{equation}\label{x2.1}
  C^{-1}t^{-\frac Q2}e^{-\frac{d(x,y)^2}{C_2t}}\le p(t,x,y)\le
  Ct^{-\frac Q2}e^{-\frac{d(x,y)^2}{C_1t}}.
\end{equation}
Moreover, the heat kernel is
proved in \cite{st1} to be a probability measure, i.e., for each $x\in X$ and
$t>0$,
\begin{equation}\label{x2.2}
T_t1(x)=\int_X p(t,x,y)\,d\mu(y)=1.
\end{equation}

The following lemma was established in \cite{krs}.

\begin{lem}\label{l2.2}
Let $T>0$. Then for $\mu$-almost every $x\in X$,
$D_yp(\cdot,x,\cdot)\in L^2([0,T]\times X)$ and there exists a
positive constant $C_{T,x}$, depending on $T$ and $x$, such that
$$\int_0^T \int_{X}|D_yp(t,x,y)|^2
\,d\mu(y)\,dt\le C_{T,x}.$$
\end{lem}

By a slight modification to the proof of \cite[Lemma 3.3]{krs},
we deduce the following estimate.
\begin{lem}\label{l2.3}
There exist $c, C>0$ such that for every $x\in X$,
$$\int_0^s \int_{2B_R(x)\setminus B_R(x)}|D_yp(t,x,y)|^2
\,d\mu(y)\,dt\le CR^{-Q/2}e^{-{cR^2}/{s}},$$
whenever $R>0$ and $s\in (0,R^2]$.
\end{lem}

\subsection{Orlicz and Zygmund spaces}
\hskip\parindent
A continuous, strictly increasing function $\Phi: [0,\fz]\to[0,\fz]$ with
$\Phi(0) = 0$ and $\Phi(\fz)=\fz$ is called an Orlicz function.
If $\Phi$ is also convex, then $\Phi$ is called a Young function.
The Orlicz space $\Phi(X)$ is then defined to be the space of all
measurable functions $f$ with $\int_X\Phi(|f|)\,d\mu<\fz.$
For $f\in \Phi(X)$, we define its Luxemburg norm as
$$\|f\|_{\Phi(X)}:=\inf\lf\{\lz>0: \,\int_X\Phi\lf(\frac{|f|}{\lz}\r)\,d\mu\le1\r\}.$$
For a Young function $\Phi$, the space $\Phi(X)$ is then a  Banach space;
see \cite{rr91}.

Functions of the type
$$\Phi_\az(t)=t\log^\az(e + t)$$
with $\az>0$ are of particular importance for us.
For such functions, the spaces $\Phi_\az(X)$ are also called Zygmund spaces.
The complementary function of $\Phi_\az$,
$\Psi_{1/\az}$, is equivalent to $\exp{t^{1/\az}}-1$.
Moreover, we have the Orlicz-H\"older inequality
\begin{equation}\label{x2.3}
\|fg\|_{L^1(X)}\le C\|f\|_{\Phi_\az(X)}\|g\|_{\Psi_{1/\az}(X)},
\end{equation}
where $C$ depends only on $Q$ and $\az$; see \cite{rr91,aikm}.

Since our aim is to prove a Moser-Trudinger type inequality,
of the form
\begin{equation*}
\fint_{B_R(y_0)} \exp\lf(c|f|\r)^{\frac {Q}{Q-1}}\,d\mu\le C,
\end{equation*}
in what follows, we modify the Orlicz function $\Psi_{\az}(t)=\exp{t^{\az}}-1$ to
the new function
$$\Psi_{R,\az}(t)=\frac{e^{t^{\az}}-1}{R^Q},$$
where $\az,R\in (0,\fz)$. Then the complementary function $\Phi_{R,1/\az}(t)$ of $\Psi_{R,\az}$
is equivalent to $t[\log(e+R^Qt)]^{1/\az}$. Moreover, $\Psi_{R,\az}$ and
$\Phi_{R,1/\az}$ satisfy the Orlicz-H\"older inequality
\begin{equation}\label{x2.4}
\|fg\|_{L^1(X)}\le C\|f\|_{\Psi_{R,\az}(X)}\|g\|_{\Phi_{R,1/\az}(X)}.
\end{equation}

\subsection{Several auxiliary results}
\hskip\parindent
We first recall the Sobolev-Poincar\'e inequalities, which follow from the
Poincar\'e inequality, see \cite{bm93,hak95,hak,sal}.
There exist positive
constants $c,C$, only depending on $C_P$ and $C_Q$, such that for all
$u\in H^{1,2}_0(B_r(x))$ with $r\le R_0$
\begin{eqnarray}\label{x2.5}
  &&\|u\|_{L^{2^\ast}(B_r(x))}\le C \||Du|\|_{L^2(B_r(x))},
\end{eqnarray}
when $Q>2$; while
\begin{eqnarray}\label{x2.6}
  &&\fint_{B_r(x)}\exp\lf(\frac{c|u|}{\||Du|\|_{L^2(B_r(x))}}\r)^2\,d\mu\le C
\end{eqnarray}
for $Q=2$; and for $Q\in (1,2)$
\begin{eqnarray}\label{x2.7}
\|u\|_{L^\fz(B_r(x))}\le Cr^{1-Q/2}\||Du|\|_{L^2(B_r(x))}.
\end{eqnarray}

\begin{lem}\label{l2.4}
Let $Q\in (1,\fz)$ and $p\in (\frac Q2,\fz]\cap (1,\fz]$.
Then there exists $C>0$ such that for all
$u\in H^{1,2}_0(B)$ and $g\in L^p(B)$ that satisfy $\Delta u=g$ in $B$,
where $B=B_R(y_0)$ with $R<R_0$,
$$\|u\|_{L^\fz(B)}\le CR^2\mu(B)^{-1/p}\|g\|_{L^p(B)}.$$
\end{lem}
\begin{proof}
We note that \cite[Theorem 4.1]{bm} states that the above inequality holds for
$p>\max\{\frac Q2,2\}$, assuming that the measure is doubling.
As the proof is similar to that of \cite[Theorem 4.1]{bm},
we here give a sketch of proof to indicate the difference of the range of $p$.

For $k\in\cn$, let $$\zeta_k(u):=\max\{u-k,0\}-\min\{u+k,0\},$$
and $A(k):=\{x\in B:\,|u|>k\}$. Then we have $\zeta_k(u)\in H^{1,2}_0(B).$
Taking a truncation argument as in
\cite[p.146]{bm}, we arrive at
$$\int_{B}|D\zeta_k(u)|^2\,d\mu\le \int_{B}g\zeta_k(u)\,d\mu.$$
Let us first assume that $Q>2$. Then by the Sobolev inequality
and the H\"older inequality, we obtain
\begin{eqnarray*}
\int_{B}|D\zeta_k(u)|^2\,d\mu&&\le \lf(\int_{A(k)}|g|^{2_\ast}\,d\mu\r)^{1/2_\ast}
\|\zeta_k(u)\|_{L^{2^\ast}(B)}\\
&&\le C\mu(A(k))^{1/2_\ast-1/p}\|g\|_{L^p(B)}\|D\zeta_k(u)\|_{L^{2}(B)},
\end{eqnarray*}
hence, $\|D\zeta_k(u)\|_{L^{2}(B)}\le C\mu(A(k))^{1/2_\ast-1/p}\|g\|_{L^p(B)}$.
Applying the Sobolev inequality again, we conclude that
\begin{eqnarray*}
\lf(\int_{B}|\zeta_k(u)|^{2^\ast}\,d\mu\r)^{1/2^\ast}&&\le
C\lf(\int_{B}|D\zeta_k(u)|^2\,d\mu\r)^{1/2}
\le C\mu(A(k))^{1/2_\ast-1/p}\|g\|_{L^p(B)}.
\end{eqnarray*}
From this inequality, we further deduce that for $h>k>0$, we have
\begin{eqnarray*}
(h-k)\mu(A(h))^{1/2^\ast}\le \lf(\int_{B}|\zeta_k(u)|^{2^\ast}\,d\mu\r)^{1/2^\ast}
\le C\mu(A(k))^{1/2_\ast-1/p}\|g\|_{L^p(B)},
\end{eqnarray*}
and hence,
\begin{eqnarray*}
\mu(A(h))\le (C\|g\|_{L^p(B)})^{2^\ast}
\frac{\mu(A(k))^{(\frac 1{2_\ast}-\frac 1p)2^\ast}}{(h-k)^{2^\ast}}.
\end{eqnarray*}
By the fact that $(\frac 1{2_\ast}-\frac 1p)2^\ast>1$ and an argument as \cite[p.147]{bm},
we conclude that $\mu(A(d))=0$, for $d=CR^2\mu(B)^{-1/p}\|g\|_{L^{p}(B)}$. Hence, we
obtain that $\|u\|_{L^\fz(B)}\le CR^2\mu(B)^{-1/p}\|g\|_{L^{p}(B)}$.

The proof of $Q=2$ is similar to the above argument, except when applying the Sobolev
inequality, we need to choose a sufficient large exponent, depending on $p$, to
substitute for $2^\ast$. We omit the details.

When $Q\in (1,2)$, by \eqref{x2.7} and the H\"older inequality, we have
\begin{eqnarray*}
\|u\|^2_{L^\fz(B)}\le CR^{2-Q}\||Du|\|^2_{L^2(B)}=CR^{2-Q}\int_B gu\,d\mu\le
CR^2\mu(B)^{-1/p}\|g\|_{L^p(B)}\|u\|_{L^\fz(B)},
\end{eqnarray*}
proving the lemma.
\end{proof}

Recall that $\Phi_{R,1/\az}(t)=t[\log(e+R^Qt)]^{1/\az}$ and $\Psi_{R,\az}(t)=\frac{1}{R^Q}(e^{t^{\az}}-1).$
\begin{lem}\label{l2.5}
Let $Q\in (1,\fz)$ and $p\in [1,\fz]$. Then there exists $C>0$, depending on $p,Q$,
such that for all
$u\in H^{1,2}_0(B)$ and $g\in L^p(B)$ that satisfy
$\Delta u=g$ in $B$, where $B=B_R(y_0)$ with $R<R_0$:

(i) when $Q>2$ and $p=2_\ast$, $\||Du|\|_{L^2(B)}\le C\|g\|_{L^{2_\ast}(B)};$

(ii) when $Q=2$, for any $p>1$, $\||Du|\|_{L^2(B)}\le C\mu(B)^{1-1/p}\|g\|_{L^p(B)};$

(iii) when $Q\in (1,2)$, $\||Du|\|_{L^2(B)}\le CR^{1-Q/2}\|g\|_{L^1(B)}.$
\end{lem}
\begin{proof}
By using the H\"older inequality and \eqref{x2.5}, we
conclude that
\begin{eqnarray*}
\int_{B} |Du(x)|^2d\mu(x)&&=-\int_{B} g(x)u(x)\,d\mu(x)\le \|g\|_{L^{2_\ast}(B)}\|u\|_{L^{2^\ast}(B)}
\le C\|g\|_{L^{2_\ast}(B)}\|Du\|_{L^{2}(B)}.
\end{eqnarray*}
Hence, $\|Du\|_{L^2(B)}\le C\|g\|_{L^{2_\ast}(B)}$,
which proves (i).

For (ii), by \eqref{x2.6}, we see that for any $q\ge 1$,
$$\|u\|_{L^q(B_r(x))}\le C\mu(B)^{1/q}\||Du|\|_{L^2(B_r(x))}. $$
From this and the H\"older inequality, we deduce that
\begin{eqnarray*}
\int_{B} |Du(x)|^2d\mu(x)&&=-\int_{B} g(x)u(x)\,d\mu(x)\le
\|g\|_{L^p(B)}\|u\|_{L^{\frac{p}{p-1}}(B)}
\le C\mu(B)^{1-1/p}\|g\|_{L^p(B)}\|Du\|_{L^2(B)},
\end{eqnarray*}
which implies $\||Du|\|_{L^2(B)}\le \mu(B)^{1/p-1}\|g\|_{L^p(B)}$.

For (iii), by \eqref{x2.7}, we have
\begin{eqnarray*}
\int_{B} |Du(x)|^2d\mu(x)&&=-\int_{B} g(x)u(x)\,d\mu(x)\le \|g\|_{L^1(B)}\|u\|_{L^\fz(B)}\le C\|g\|_{L^1(B)}R^{1-\frac Q2}\|Du\|_{L^2(B)}
\end{eqnarray*}
proving the lemma.
\end{proof}

\begin{lem}\label{l2.6}
Let $Q\in (1,\fz)$, $p\in (\frac Q2,\fz]\cap (1,\fz]$ and $B=B_R(y_0)$ with
$R<R_0$. For every $g\in L^p(B)$, there exists $u\in H^{1,2}_0(B)$ such that
$\Delta u=g$ in $B$.
\end{lem}
\begin{proof}
For each $k\in \cn$, let $g_k=g\chi_{B\cap \{|g|\le k \}}$.
Then by \cite[p.131]{bm93}, there exists $u_k\in H^{1,2}_0(B)$ such that
$\Delta u_k=g_k$ in $B$.
Moreover, by Lemma \ref{l2.4} and Lemma \ref{l2.5}, we have
\begin{eqnarray*}
\|u_k-u_j\|_{L^2(B)}+\||D(u_k-u_j)|\|_{L^2(B)}&&\le C_R\|g_k-g_j\|_{L^p(B)}\to 0,
\end{eqnarray*}
as $k,j\to\fz$. Hence $\{u_k\}_{k\in\cn}$ is a Cauchy sequence in $H^{1,2}_0(B)$,
and there exists $ u\in H^{1,2}_0(B)$ such that $\lim_{k\to\fz}u_k=u$ in
$H^{1,2}_0(B)$. Moreover, for each $\phi\in H^{1,2}_0(B)$,
we have
\begin{eqnarray*}
-\int_{B} Du(x)\cdot D\phi(x)\,d\mu(x)&&=-\lim_{k\to\fz}\int_{B}Du_k(x)\cdot D\phi(x)\,d\mu(x)\\
&&=\lim_{k\to\fz} \int_B g_k(x)\phi(x)\,d\mu(x)=\int_{B} g(x) \phi(x)\,d\mu(x),
\end{eqnarray*}
proving the lemma.
\end{proof}

Combining Lemma \ref{l2.4} and Lemma \ref{l2.6},
we deduce the following estimate.
\begin{lem}\label{l2.8} Let $Q\in (1,\fz)$ and  $p\in(\frac Q2,\fz]\cap (1,\fz]$.
Then there exists a positive constant $C$ such that for all
$u\in H^{1,2}_{\loc}(X)$ and $g\in L^p_\loc(X)$ that satisfy
$\Delta u=g$ in $2B$, where $B=B_R(y_0)$ with $R< R_0/2$,
\begin{eqnarray*}
\|u\|_{L^\fz(B)}\le C[R^{-Q/2}\|u\|_{L^2(2B)}
+R^{2-Q/p}\|g\|_{L^p(2B)}].
\end{eqnarray*}
\end{lem}
\begin{proof}
by Lemma \ref{l2.6}, there exists
$\wz u \in H_0^{1,2}(2B)$ such that $\Delta \wz u=g$ in $2B$.
Then from Lemma \ref{l2.4}, we deduce that
\begin{equation*}
\|\wz u\|_{L^\fz(2B)}\le CR^2\mu(B)^{-1/p}\|g\|_{L^p(2B)}.
\end{equation*}
Now $u-\wz u$ is Cheeger-harmonic in $2B$, which together with
\cite[Theorem 5.4]{bm} implies that
\begin{eqnarray*}
\|u-\wz u\|_{L^\fz(B)}\le CR^{-Q/2}\|u-\wz u\|_{L^2(2B)}.
\end{eqnarray*}
The above two estimates give the desired results.
\end{proof}

We also need the H\"older continuity of the solutions.
\begin{lem}\label{l2.7}
Let $Q\in (1,\fz)$ and $p\in (\frac Q2,\fz]\cap (1,\fz]$.
Then there exist $C>0$ and $\gz\in (0,1)$  such that for all
$u\in H^{1,2}_\loc(X)$ and $g\in L^p_\loc(X)$ that satisfy
$\Delta u=g$ in $4B$, where $B=B_R(y_0)$ with $R<R_0/4$,
and almost all $x,y\in B$,
$$|u(x)-u(y)|\le C\lf\{R^{-Q/2}\|u\|_{L^2(4B)}
+R^{2-Q/p}\|g\|_{L^p(4B)}\r\}\lf(\frac{d(x,y)}{R}\r)^\gz.$$
\end{lem}
\begin{proof}
Let $M_2=\sup_{B_{2R}(y_0)}u$, $m_2=\inf_{B_{2R}(y_0)}u$,
$M_1=\sup_{B_R(y_0)}u$ and $m_1=\inf_{B_R(y_0)}u$. By Lemma \ref{l2.6},
there exists $\wz u\in H^{1,2}_{0}(B_{2R}(y_0))$ such that $\Delta \wz u=g$
in $B_{2R}(y_0)$.

Let $M_R=\|\wz u\|_{L^\fz(B_{2R}(y_0))}$.
Applying \cite[Theorem 1.1]{bm} to $M_2+M_R-(u-\wz u)$
and  $(u-\wz u)-m_2+M_R$ respectively,
we obtain that
\begin{eqnarray*}
M_2-m_1\le\sup_{B_R(y_0)}M_2+M_R-(u-\wz u)\le C_3\inf_{B_R(y_0)} \lf[M_2+M_R-(u-\wz u)\r]\le
C_3\lf[M_2-M_1+2M_R\r],
\end{eqnarray*}
\begin{eqnarray*}
M_1-m_2\le\sup_{B_R(y_0)}(u-\wz u)-m_2+M_R\le C_3\inf_{B_R(y_0)} \lf[(u-\wz u)-m_2+M_R\r]\le
C_3\lf[m_1-m_2+2M_R\r].
\end{eqnarray*}
Adding the last two inequalities, we deduce that
$$(C_3+1)(M_1-m_1)\le (C_3-1)(M_2-m_2)+4C_3M_R.$$
By Lemma \ref{l2.4}, we conclude that for each $p\in (\frac Q2,\fz]\cap (1,\fz]$,
\begin{eqnarray*}
{\mathrm {osc}}({u,{B_R(y_0)}})\le \frac{C_3-1}{C_3+1} {\mathrm {osc}}({u,{B_{2R}(y_0)}})
+CR^{2-Q/p}\|g\|_{L^p(B_{2R}(y_0))},
\end{eqnarray*}
which together with a standard iteration as in \cite[p.201]{gt01} and Lemma \ref{l2.8}
yields the desired estimate.
\end{proof}

By Lemma \ref{l2.8}, similarly to the proof of \cite[Lemma 2.2]{ji}, we have
the following Caccioppoli inequality.
\begin{lem}\label{l2.9}
Let $Q\in (1,\fz)$ and $p\in(\frac Q2,\fz]\cap (1,\fz]$.
Then there exists a positive constant $C$ such that
for all $u\in H^{1,2}_{\loc}(X)$ and $g\in L^p_\loc(X)$ that satisfy $\Delta u=g$ in $B_R(y_0)$,
where $r<R<R_0$,
\begin{equation*}
\||Du|\|_{L^2(B_r(y_0))}\le
CR^{1+Q(\frac 12-\frac 1p)}\|g\|_{L^p(B_{R}(y_0))}+\frac{C}{(R-r)}\|u\|_{L^2(B_{R}(y_0))}.
\end{equation*}
\end{lem}

\section{Poisson equation}
\hskip\parindent
Let $B=B_R(y_0)\subset\Omega$ satisfy $8B\subset\subset\Omega$.
Let  $\psi$ be a Lipschitz function such that
$\psi=1$ on $B_{2R}(y_0)$, $\supp \psi\subset B_{4R}(y_0)$ and
$|D\psi|\le\frac {C_4}R$.
For all $x,x_0\in 8B$, set $w_{x_0}(t,x):=u\psi(x)-T_t(u\psi)(x_0)$.
Then $Dw_{x_0}(t,x_0)=D(u\psi)(x_0)=Du(x_0)$ for every $x_0\in B_{2R}(y_0)$.

The following functional is the main tool for us; see \cite{ck,krs,ji}.
Let $x_0\in B=B_R(y_0)$. For all $t\in (0,R^2)$, define
\begin{eqnarray}\label{x3.1}
J(t):=&&\frac{1}{t}
\bigg\{\int_0^t\int_X |Dw_{x_0}(s,x)|^2 p(s,x_0,x)\,d\mu(x)\,ds\nonumber\\
&&+\int_0^t\int_X w_{x_0}(s,x)\psi(x)Au(x)p(s,x_0,x)\,d\mu(x)\,ds\bigg\}.
\end{eqnarray}

The main aim of this section is to prove the following estimate.
\begin{thm}\label{t3.1}
Let $Q\in (1,\fz)$, $p\in (\frac Q2,\fz]\cap (1,\fz]$ and assume that
\eqref{1.1} and the curvature condition \eqref{1.3} hold.
Then there exists $C>0$ such that for all $u\in H^{1,2}_\loc(X)$
and $g\in L^p_\loc(X)$ that satisfy $\Delta u=g$ in $8B$,
where $B=B_R(y_0)$ with  $R<R_0/8$, and almost every $x_0\in B$,
\begin{eqnarray}\label{x3.2}
|Du(x_0)|^2&&\le C(1+ c_{\kz}(R^2)R^2)C(u,g)^2+
\int_0^{R^2}\frac{1}{t} \int_X \lf|w_{x_0}(t,x)\psi(x)g(x)\r|p(t,x_0,x)\,d\mu(x)\,dt,
\end{eqnarray}
where $C(u,g)=R^{-Q/2-1}\|u\|_{L^2(8B)}+R^{1-Q/p}\|g\|_{L^p(8B)}$.
\end{thm}

\begin{rem}\rm
In this paper, the curvature condition \eqref{1.3} is only employed once,
in the proof of Theorem \ref{t3.1}; see the proof at the end of this section.
\end{rem}

Notice that $w_{x_0}(0,x_0)=0.$ We use the H\"older continuity of $u$ to
obtain the H\"older continuity of $w_{x_0}(t,x)$ at $(0,x_0)$.

\begin{lem}\label{l3.1}
Let $Q\in (1,\fz)$ and $p\in (\frac Q2,\fz]\cap (1,\fz]$.
Then there exists $C>0$ such that for all
$u\in H^{1,2}_\loc(X)$ and $g\in L^p_\loc(X)$ that satisfy $\Delta u=g$ in $8B$,
where $B=B_R(y_0)$ with $R<R_0/8$, and
almost all $x_0\in B$, $x\in 2B$ and all $t\in (0,R^2)$,
$$|w_{x_0}(t,x)|=|u\psi(x)-T_t(u\psi)(x_0)|\le CC(u,g)R^{1-\gz}(d(x,x_0)^\gz+t^{\gz/2}),$$
where $C(u,g)=R^{-Q/2-1}\|u\|_{L^2(8B)}
+R^{1-Q/p}\|g\|_{L^p(8B)}$ and $\gz\in (0,1)$ is
as in Lemma \ref{l2.7}.
\end{lem}
\begin{proof}
In the following proof, we will repeatedly use the fact
that for fixed $\bz,\dz\in (0,\fz)$, $t^{\bz}e^{-t^\dz}$ and $t^{-\bz}e^{-t^{-\dz}}$
are bounded on $(0,\fz)$.

By Lemma \ref{l2.7}, we see that for almost all $x_0,x\in 2B$,
$$|u(x)-u(x_0)|\le CRC(u,g)\lf(\frac{d(x,x_0)}{R}\r)^\gz ,$$
where $C$ and $\gz$ are independent of $u,g$ and $B$.
Thus for almost all
$x_0\in B$, $x\in 2B$ and all $t\in (0,R^2)$, by Lemma \ref{l2.8}, we have
\begin{eqnarray*}
&&|w_{x_0}(t,x)|=|u(x)\psi(x)-T_t(u\psi)(x_0)|\nonumber\\
&&\hs=|u(x)\psi(x)-u(x_0)\psi(x_0)+u(x_0)\psi(x_0)-T_t(u\psi)(x_0)|\nonumber\\
&&\hs \le CC(u,g)R^{1-\gz}d(x,x_0)^\gz+\int_{2B} |u(x_0)\psi(x_0)-u(y)\psi(y)|p(t,x_0,y)\,d\mu(y)\nonumber\\
&&\hs\hs+\int_{X\setminus 2B} |u(x_0)\psi(x_0)-u(y)\psi(y)|p(t,x_0,y)\,d\mu(y)\nonumber\\
&&\hs \le CC(u,g)R^{1-\gz}d(x,x_0)^\gz+CC(u,g)R^{1-\gz}\int_{2B} d(y,x_0)^\gz t^{-\frac Q2}
e^{-\frac{d(y,x_0)^2}{2C_1t}} e^{-\frac{d(y,x_0)^2}{2C_1t}}\,d\mu(y)\nonumber\\
&&\hs\hs+e^{-cR^2/t}\|u\|_{L^\fz(4B)}\int_{X\setminus 2B}t^{-\frac Q2}
 e^{-\frac{d(y,x_0)^2}{2C_1t}}\,d\mu(y)\nonumber \\
&&\hs \le CRC(u,g)\lf[R^{-\gz}(d(x,x_0)^\gz+t^{\gz/2})+e^{{-cR^2/t}}\r]
\int_X  (lt)^{-\frac Q2}e^{-\frac{d(y,x_0)^2}{C_2 (lt)}}\,d\mu(y)\nonumber\\
&&\hs \le CRC(u,g)\lf[R^{-\gz}(d(x,x_0)^\gz+t^{\gz/2})\r]\int_X  p(lt,x_0,x)\,d\mu(x)\nonumber\\
&&\hs\le CRC(u,g)\lf[R^{-\gz}(d(x,x_0)^\gz+t^{\gz/2})\r],
\end{eqnarray*}
where $l=\frac{2C_1}{C_2}$, as desired.
\end{proof}

The following result shows the motivation for using the functional $J$.

\begin{prop}\label{p3.1}
Let $Q\in (1,\fz)$, $p\in (\frac Q2,\fz]\cap (1,\fz]$ and $B=B_R(y_0)$ with $R<R_0/8$.
Suppose that $u\in H^{1,2}_\loc(X)$ and $g\in L^p_\loc(X)$ satisfy $\Delta u=g$ in $8B$.
Then, for almost every $x_0\in B$,
$\lim_{t\to 0^+}J(t)=|Du(x_0)|^2.$
\end{prop}
\begin{proof}
By Lemma \ref{l2.2}, for almost every $x_0\in B$,
$D_yp(s,x_0,\cdot)\in L^2(X)$. From this together with the fact that
for almost every $s$, $w_{x_0}(s,\cdot), p(s,x_0,\cdot)$ are bounded functions and belong
in $H^{1,2}_{\loc}(X)$, $\supp \psi\subset 4B$, we see that $w_{x_0} \psi p\in H^{1,2}_0(B(y_0,4R))$.
Thus, we conclude that
\begin{eqnarray}\label{x3.3}
\quad \int_0^t\int_X w_{x_0}(s,x)\psi(x)Au(x)p(s,x_0,x)\,d\mu(x)
=\int_0^t\int_{4B} w_{x_0}(s,x)\psi(x)g(x)p(s,x_0,x)\,d\mu(x).
\end{eqnarray}
By Lemma \ref{l3.1}, $|w_{x_0}(s,x)|\le CC(u,g)R^{1-\gz}(d(x_0,x)^\gz+s^{\gz/2})$ for some $\gz\in (0,1)$
and almost every $x\in 2B$. This further implies that
\begin{eqnarray*}
&&\lf|\int_0^t\int_X w_{x_0}(s,x)p(s,x_0,x)\psi(x)Au(x)\,d\mu(x)\,ds\r|\nonumber\\
&&\hs\le CC(u,g)R^{1-\gz}\int_0^t\int_{2B}
(d(x,x_0)^\gz+s^{\gz/2})s^{-\frac Q2}e^{-\frac{d(x,x_0)^2}{C_1s}}
\lf|g(x)\r|\,d\mu(x)\,ds\nonumber\\
&&\hs\hs+C\|u\|_{L^\fz(4B)}\int_0^t\int_{4B\setminus 2B}
s^{-\frac Q2}e^{-cR^2/s}
\lf|g(x)\r|\,d\mu(x)\,ds\nonumber\\
&&\hs\le CC(u,g)R^{1-\gz}\int_0^t s^{\gz/2}\int_{X}
s^{-\frac Q2}e^{-\frac{d(x,x_0)^2}{2C_1s}}
\lf|g(x)\r|\,d\mu(x)\,ds+Ct^2\|u\|_{L^\fz(4B)}R^{-Q-2}\|g\|_{L^1(4B)}\\
&&\hs\le CC(u,g)R^{1-\gz}\int_0^t s^{\gz/2}T_{ls}(|g|)(x_0)\,ds+
Ct^2\|u\|_{L^\fz(4B)}R^{-Q-2}\|g\|_{L^1(4B)},
\end{eqnarray*}
where $l=\frac{C_1}{2C_2}$. By the fact that $T_t-I\to 0$ in the strong
operator topology as $t\to 0$, we obtain
\begin{eqnarray}\label{x3.4}
&&\lim_{t\to 0^+}\lf|\frac{1}{t}
\int_0^t\int_X w_{x_0}(s,x)p(s,x_0,x)\psi(x)Au(x)\,d\mu(x)\,ds\r|\nonumber\\
&&\hs\le \lim_{t\to 0^+} \lf\{CC(u,g)R^{1-\gz}\frac 1t\int_0^t s^{\gz/2}T_{ls}(|g|)(x_0)\,ds+
Ct\|u\|_{L^\fz(4B)}R^{-Q-2}\|g\|_{L^1(4B)}\r\}\nonumber\\
&&\hs=CC(u,g)R^{1-\gz}\lim_{s\to 0^+}s^{\gz/2}T_{ls}(|g|)(x_0)=0,
\end{eqnarray}
for almost every $x_0\in B_R(y_0)$, which implies that
\begin{eqnarray*}
\lim_{t\to 0^+}J(t)&&=\lim_{s\to 0^+} T_s(|D(u\psi)|^2)(x_0)=|Du(x_0)|^2
\end{eqnarray*}
for almost every $x_0\in B_R(y_0)$, proving the proposition.
\end{proof}

By Lemma \ref{l3.1}, similarly to \cite[(24)]{krs} and \cite[(3.5)]{ji},
we deduce the following equality. We omit the details.
\begin{lem}\label{l3.2}
Let $Q\in (1,\fz)$, $p\in (\frac Q2,\fz]\cap (1,\fz]$ and $B=B_R(y_0)$ with $R<R_0/8$.
Suppose that $u\in H^{1,2}_\loc(X)$ and $g\in L^p_\loc(X)$ that satisfy $\Delta u=g$ in $8B$.
Then for almost every $x\in B$ and all $t\in (0,R^2)$,
\begin{eqnarray*}
&&\int_0^t\int_X \lf(A+\frac{\pa}{\pa s}\r)
w_{x_0}^2(s,x)p(s,x_0,x)\,d\mu(x)\,ds=\int_X w_{x_0}^2(t,x)p(t,x_0,x)\,d\mu(x).
\end{eqnarray*}
\end{lem}

We now begin to estimate the functional $J(t)$.
\begin{prop}\label{p3.2}
Let $Q\in (1,\fz)$ and $p\in (\frac Q2,\fz]\cap (1,\fz]$.
Then there exists $C>0$ such that for all
$u\in H^{1,2}_\loc(X)$ and $g\in L^p_\loc(X)$ that satisfy $\Delta u=g$ in $8B$,
where $B=B_R(y_0)$ with $R< R_0/8$, and almost every $x_0\in B$,
\begin{equation*}
J(R^2)\le C\lf(\frac{\|u\|_{L^2(8B)}}{R^{Q/2+1}}+\frac{\|g\|_{L^p(8B)}}{R^{Q/p-1}}\r)^2.
\end{equation*}
\end{prop}
\begin{proof}
Since $w_{x_0}(t,x)=u(x)\psi(x)-T_t(u\psi)(x_0),$ we have
$$|D(u\psi)|^2=|Dw_{x_0}|^2=\frac 12 A w_{x_0}^2-w_{x_0}(\psi Au+uA\psi+2Du\cdot D\psi)$$
in the weak sense of measures. Also, in what follows we extend $A$ formally
to all of $H^{1,2}(X)$ by defining
\begin{equation*}
\int_X v(x) Au(x)\,d\mu(x)=-\int_X Dv(x)\cdot Du(x)\,d\mu(x)=\int_X Av(x) u(x)\,d\mu(x).
\end{equation*}

Moreover, we set $m(t)=T_t(u\psi)(x_0)$. Then
$\frac{\pa}{\pa t}w_{x_0}^2=2w_{x_0}\frac{\pa}{\pa t}w_{x_0}=-2w_{x_0}m'(t)$, which further implies that
$$|Dw_{x_0}|^2=\frac 12 \lf(A+\frac{\pa}{\pa t}\r)w_{x_0}^2-w_{x_0}(\psi Au+uA\psi+2Du\cdot D\psi-m'(t))$$
in the weak sense of measures.
Thus, we obtain
\begin{eqnarray}\label{x3.5}
&&\int_0^t\int_X |Dw_{x_0}(s,x)|^2 p(s,x_0,x)\,d\mu(x)\,ds\nonumber\\
&&\hs =\frac 12 \int_0^t\int_X \lf(A+\frac{\pa}{\pa s}\r)
w^2_{x_0}(s,x)p(s,x_0,x)\,d\mu(x)\,ds\nonumber\\
&&\hs\hs-\int_0^t\int_X w_{x_0}(s,x)[\psi Au+uA\psi+2Du\cdot D\psi-m'(s)]p(s,x_0,x)\,d\mu(x)\,ds.
\end{eqnarray}
Recall that for each $s>0$ and $x_0\in X$, $T_s(1)(x_0)=1$.
We then have
\begin{eqnarray*}
&&\int_0^t\int_X w_{x_0}(s,x)m'(s)p(s,x_0,x)\,d\mu(x)\,ds
=\int_0^t\int_X m'(s)T_s(u\psi)(x_0)\lf(1-T_s(1)(x_0)\r)\,ds=0.
\end{eqnarray*}

We now estimate the second term in \eqref{x3.5}. Recall that
$\psi=1$ on $2B=2B_R(y_0)$ and $\supp \psi \subseteq 4B$.
By  Lemma \ref{l2.8}, Lemma \ref{l2.9}, Lemma \ref{l2.3} and the H\"older
inequality, we obtain
\begin{eqnarray*}
&&\lf|\int_0^t\int_X w_{x_0}(s,x)u(x) A\psi(x)p(s,x_0,x)\,d\mu(x)\,ds\r|\nonumber\\
&&\hs=\lf|\int_0^t\int_X D(w_{x_0}(s,\cdot)up(s,x_0,\cdot))(x)\cdot D\psi(x)
\,d\mu(x)\,ds\r|\nonumber\\
&&\hs\le Ct^{1/2}R^{-1+\frac Q2}\|u\|_{L^\fz(4B)}^2
\lf(\int_0^t\int_{5B_R(x_0)\setminus B_R(x_0)} |Dp(s,x_0,x)|^2
\,d\mu(x)\,ds\r)^{1/2}\nonumber\\
&&\hs\hs+CtR^{-1+\frac Q2} t^{-\frac Q2}e^{-\frac{R^2}{ct}}\|u\|_{L^\fz(4B)}
\lf(\int_{4B\setminus 2B}(|Du(x)|^2+|D(u\psi)(x)|^2)\,d\mu(x)\r)^{1/2}\nonumber\\
&&\hs\le C t^{1/2}R^{-1}e^{-\frac{R^2}{ct}}\|u\|^2_{L^\fz(4B)}+
CtR^{-1-\frac Q2}e^{-\frac{R^2}{ct}}\|u\|_{L^\fz(4B)}\|Du\|_{L^2(4B)}\nonumber\\
&&\hs\le Cte^{-cR^2/t}(R^{-1-\frac Q2}\|u\|_{L^2(8B)}+R^{1-\frac Qp}\|g\|_{L^p(8B)})^2.
\end{eqnarray*}
Similarly, we have
\begin{eqnarray*}
&&\lf|\int_0^t\int_X w_{x_0}(s,x)p(s,x_0,x)Du(x)\cdot D\psi(x)\,d\mu(x)\,ds\r|\nonumber\\
&&\hs\le Cte^{-cR^2/t}(R^{-1-\frac Q2}\|u\|_{L^2(8B)}+R^{1-\frac Qp}\|g\|_{L^p(8B)})^2.
\end{eqnarray*}

Combining the above estimates, by \eqref{x3.5} and Lemma \ref{l3.2}, we obtain that
\begin{eqnarray}\label{x3.6}
tJ(t)&&\le \frac 12 \lf|\int_0^t\int_X \lf(A+\frac{\pa}{\pa s}\r)
w_{x_0}^2(s,x)p(s,x_0,x)\,d\mu(x)\,ds\r|\nonumber\\
&&\hs+\lf|\int_0^t\int_X w_{x_0}(s,x)[u(x) A\psi(x)
+2Du(x)\cdot D\psi(x)]p(s,x_0,x)\,d\mu(x)\,ds\r|\nonumber\\
&& \le \frac 12 \int_X w_{x_0}^2(t,x)p(t,x_0,x)\,d\mu(x)
+Cte^{-cR^2/t}\lf(\frac{\|u\|_{L^2(8B)}}{R^{Q/2+1}}+\frac{\|g\|_{L^p(8B)}}{R^{Q/p-1}}\r)^2.
\end{eqnarray}

Hence, by Lemma \ref{l2.8} again, we conclude that
\begin{eqnarray*}
J(R^2)&&\le \frac 1{2R^2}\int_X w_{x_0}^2(R^2,x)p(R^2,x_0,x)\,d\mu(x)
+C\lf(\frac{\|u\|_{L^2(8B)}}{R^{Q/2+1}}+\frac{\|g\|_{L^p(8B)}}{R^{Q/p-1}}\r)^2
\nonumber\\
&&\le \frac 1{2R^2}\|u\|^2_{L^\fz(4B)}\int_Xp(R^2,x_0,x)\,d\mu(x)
+C\lf(\frac{\|u\|_{L^2(8B)}}{R^{Q/2+1}}+\frac{\|g\|_{L^p(8B)}}{R^{Q/p-1}}\r)^2\nonumber\\
&&\le C\lf(\frac{\|u\|_{L^2(8B)}}{R^{Q/2+1}}+\frac{\|g\|_{L^p(8B)}}{R^{Q/p-1}}\r)^2,
\end{eqnarray*}
which completes the proof of Proposition \ref{p3.2}.
\end{proof}

We use the H\"older continuity (Lemma \ref{l3.1}) of $w_{x_0}(t,x)$ to deduce the following estimate.
\begin{prop}\label{p3.3}
Let $Q\in (1,\fz)$ and $p\in (\frac Q2,\fz]\cap (1,\fz]$.
Then there exists $C>0$ such that for all
$u\in H^{1,2}_\loc(X)$ and $g\in L^p_\loc(X)$ that satisfy $\Delta u=g$ in $8B$,
where $B=B_R(y_0)$ with $R<R_0/8$, and almost every $x_0\in B$,
\begin{eqnarray*}
&&\int_0^{R^2} \frac 1t\int_X w_{x_0}^2(t,x)p(t,x_0,x)\,d\mu(x)\,dt\le CR^2C(u,g)^2,
\end{eqnarray*}
where $C(u,g)=R^{-Q/2-1}\|u\|_{L^2(8B)}+R^{1-Q/p}\|g\|_{L^p(8B)}$.
\end{prop}
\begin{proof}

By Lemma \ref{l3.1}, we deduce that
\begin{eqnarray*}
&&\int_X w_{x_0}^2(t,x)p(t,x_0,x)\,d\mu(x)\\
&&\hs= \int_{2B}w_{x_0}^2(t,x)p(t,x_0,x)\,d\mu(x)
+\int_{X\setminus 2B}w_{x_0}^2(t,x)p(t,x_0,x)\,d\mu(x)\nonumber\\
&&\hs \le C[C(u,g)R^{1-\gz}]^2\int_{2B}(d(x,x_0)^\gz+t^{\gz/2})^2t^{-\frac Q2}e^{-\frac{d(x,x_0)^2}{2C_1t}}
e^{-\frac{d(x,x_0)^2}{2C_1t}}\,d\mu(x)\nonumber\\
&&\hs\hs+C\|u\|_{L^\fz(4B)}^2\int_{X\setminus 2B}
t^{-\frac Q2}e^{-\frac{d(x,x_0)^2}{2C_1t}}
e^{-\frac{d(x,x_0)^2}{2C_1t}}\,d\mu(x)\nonumber\\
&&\hs \le C[C(u,g)R^{-\gz}]^2t^{\gz}\int_{2B}p(lt,x_0,x)\,d\mu(x)+Ce^{-cR^2/t}\|u\|_{L^\fz(4B)}^2
\int_{X\setminus 2B}p(lt,x_0,x)\,d\mu(x)\nonumber\\
&&\hs\le CR^2C(u,g)^2[R^{-2\gz}t^\gz+e^{-cR^2/t}]\int_{X}p(lt,x_0,x)\,d\mu(x)\\
&&\hs\le CR^2C(u,g)^2R^{-2\gz}t^\gz,
\end{eqnarray*}
where $l=\frac{2C_1}{C_2}$ and we used the fact that $e^{-cR^2/t}\le C(\frac t{R^2})^\gz$. From this, we further conclude that
\begin{eqnarray*}
&&\int_0^{R^2} \frac 1t\int_X w_{x_0}^2(t,x)p(t,x_0,x)\,d\mu(x)\le \int_0^{R^2}
CC(u,g)^2R^{2-2\gz}t^{\gz-1}\,dt\le CR^2C(u,g)^2,
\end{eqnarray*}
which completes the proof of Proposition \ref{p3.3}.
\end{proof}

We are now in position to prove the main result of this section.

\begin{proof}[{\bf Proof of Theorem \ref{t3.1}}]
 Let us first estimate the derivative $J'(t)=\frac{\,d}{\,dt}J(t)$.
By \eqref{x3.3}, \eqref{x3.1} and \eqref{x3.6}, we deduce that
\begin{eqnarray*}
\frac{\,d}{\,dt}J(t)&&=-\frac{1}{t^2}J(t)+\frac{1}{t}
\int_X |Dw_{x_0}(t,x)|^2 p(t,x_0,x)\,d\mu(x)\nonumber\\
&&\hs\hs+\frac{1}{t}\int_X w_{x_0}(t,x)\psi(x)g(x)p(t,x_0,x)\,d\mu(x)\nonumber\\
&&\hs\ge \frac{1}{t}\lf(\int_X |Dw_{x_0}(t,x)|^2 p(t,x_0,x)\,d\mu(x)-
\frac 1{2t}\int_X w_{x_0}^2(t,x)p(t,x_0,x)\,d\mu(x)\r)\nonumber\\
&&\hs\hs-\frac Cte^{-cR^2/t}C(u,g)^2
+\frac{1}{t} \int_X w_{x_0}(t,x)\psi(x)g(x)p(t,x_0,x)\,d\mu(x).
\end{eqnarray*}

For each fixed $t\in (0,R^2)$, either
$$\int_X |Dw_{x_0}(t,x)|^2 p(t,x_0,x)\,d\mu(x)\ge
\frac 1{2t}\int_X w_{x_0}^2(t,x)p(t,x_0,x)\,d\mu(x)$$
or
$$\int_X |Dw_{x_0}(t,x)|^2 p(t,x_0,x)\,d\mu(x)<
\frac 1{2t}\int_X w_{x_0}^2(t,x)p(t,x_0,x)\,d\mu(x).$$
In the first case, we have
\begin{eqnarray}\label{x3.7}
\frac{\,d}{\,dt}J(t)\ge&& -\frac Cte^{-cR^2/t}C(u,g)^2
-\frac{1}{t} \int_X w_{x_0}(t,x)\psi(x)g(x)p(t,x_0,x)\,d\mu(x).
\end{eqnarray}
In the second case, by the curvature condition \eqref{1.3} with $T=R^2$, we deduce that
\begin{eqnarray}\label{x3.8}
\frac{\,d}{\,dt}J(t)&&\ge-c_{\kz}(R^2)\int_X |Dw_{x_0}(t,x)|^2 p(t,x_0,x)\,d\mu(x)
-\frac Cte^{-cR^2/t}C(u,g)^2\nonumber\\
&&\hs+\frac{1}{t} \int_X w_{x_0}(t,x)\psi(x)g(x)p(t,x_0,x)\,d\mu(x)\nonumber\\
&&\ge-\frac{c_{\kz}(R^2)}{2t}\int_X w_{x_0}^2(t,x)p(t,x_0,x)\,d\mu(x)
-\frac Cte^{-cR^2/t}C(u,g)^2\nonumber\\
&&\hs+\frac{1}{t} \int_X w_{x_0}(t,x)\psi(x)g(x)p(t,x_0,x)\,d\mu(x).
\end{eqnarray}
From \eqref{x3.7} and \eqref{x3.8},
we see that \eqref{x3.8} holds in both cases.
Integrating over $(0,R^2)$ and applying Proposition \ref{p3.3} we conclude that
\begin{eqnarray*}
\int_0^{R^2}J'(t)\,dt&&\ge-\int_0^{R^2}\lf\{\frac{c_{\kz}(R^2)}{2t}\int_X w_{x_0}^2(t,x)p(t,x_0,x)\,d\mu(x)
-\frac Cte^{-cR^2/t}C(u,g)^2\r\}\,dt\nonumber\\
&&\hs+\int_0^{R^2}\frac{1}{t} \int_X w_{x_0}(t,x)\psi(x)g(x)p(t,x_0,x)\,d\mu(x)\,dt\\
&&\ge -C(1+c_{\kz}(R^2)R^2)C(u,g)^2+\int_0^{R^2}\frac{1}{t}
\int_X w_{x_0}(t,x)\psi(x)g(x)p(t,x_0,x)\,d\mu(x)\,dt.
\end{eqnarray*}
Combining Proposition \ref{p3.1} and Proposition \ref{p3.2}, we obtain that for almost every $x_0\in B$,
\begin{eqnarray*}
|Du(x_0)|^2&&=J(R^2)-\int_0^{R^2} \frac{\,d}{\,dt}J(t)\,dt\\
&&\le C(1+ c_{\kz}(R^2)R^2)C(u,g)^2+\bigg|\int_0^{R^2}\frac{1}{t}
\int_X w_{x_0}(t,x)\psi(x)g(x)p(t,x_0,x)\,d\mu(x)\,dt\bigg|,
\end{eqnarray*}
which completes the proof of Theorem \ref{t3.1}.
\end{proof}

We end this section by using Theorem \ref{t3.1} to obtain an $L^\fz$-estimate for
$|Du|$ when $g\in L^\fz$.
\begin{lem}\label{l3.3}
Let $Q\in (1,\fz)$ and $B=B_R(y_0)$ with $R<R_0/8$.
Suppose that $u\in H^{1,2}_\loc(X)$ and $g\in L^\fz_\loc(X)$ that satisfy $\Delta u=g$ in $8B$.
Then $\||D u|\|_{L^\fz(B)}<\fz.$
\end{lem}
\begin{proof} By Theorem \ref{t3.1}, we have that for almost every $x_0\in B$,
\begin{eqnarray*}
|Du(x_0)|^2&&\le C(1+ c_{\kz}(R^2)R^2)C(u,g)^2+
\bigg|\int_0^{R^2}\frac{1}{t} \int_X w_{x_0}(t,x)\psi(x)g(x)p(t,x_0,x)\,d\mu(x)\,dt\bigg|,
\end{eqnarray*}
where $C(u,g)=R^{-Q/2-1}\|u\|_{L^2(8B)}+R\|g\|_{L^\fz(8B)}$.
Applying Lemma \ref{l3.1}, similarly to the proof of Proposition \ref{p3.3},
we further deduce that
\begin{eqnarray*}
|Du(x_0)|^2&&\le C(1+ c_{\kz}(R^2)R^2)C(u,g)^2+CC(u,g)\|g\|_{L^\fz(8B)},
\end{eqnarray*}
which implies that $\||Du|\|_{L^\fz(B)}<\fz$, proving the lemma.
\end{proof}

\section{Auxiliary equations}
\hskip\parindent  Suppose that $\Delta u=g$ in $8B$. From Section 3,
we have the following pointwise boundedness of $|Du|$: for almost every $x_0\in B$,
\begin{eqnarray*}
|Du(x_0)|^2&&\le C(1+ c_{\kz}(R^2)R^2)C(u,g)^2+
\int_0^{R^2}\frac{1}{t} \int_X \lf|w_{x_0}(t,x)\psi(x)g(x)\r|p(t,x_0,x)\,d\mu(x)\,dt,
\end{eqnarray*}
where $C(u,g)=R^{-Q/2-1}\|u\|_{L^2(8B)}+R^{1-Q/p}\|g\|_{L^p(8B)}$ and
$p\in (\frac Q2,\fz]\cap (1,\fz]$.
Hence, the main problem left is to estimate the second term on the right-hand side.
We do not know how to estimate it for general $g$, but we can estimate it provided
that we assume that the support of $g$ is contained in $\lz B$ for some $\lz\in (0,1)$.

Thus, in this section, we study the auxiliary equation
that for a ball $B=B_R(y_0)$ with $R<R_0/8$,
\begin{equation*}
-\int_{8B} Du(x)\cdot D\phi(x)\,d\mu(x)=\int_{8B} g(x)\phi(x)\,d\mu(x),
\ \ \forall \phi\in H_0^{1,2}(8B),
\end{equation*}
where $u\in H^{1,2}_0(8B)$ and $g\in L^\fz(X)$ with $\supp g\subset B/4$.

The main aim of this section is to prove Theorem \ref{t1.1} and Theorem
\ref{t1.2} when $u$ and $g$ are as above.

\begin{thm}\label{t4.1}
Let $Q\in (1,\fz)$ and suppose that \eqref{1.1} and \eqref{1.3} hold.
Then there exists $c,C>0$ such that for all $u\in H^{1,2}_0(8B)$ and
$g\in L^\fz(X)$ with $\supp g\subset B/4$ that satisfy
$\Delta u=g$ in $8B$, where $B=B_R(y_0)$ with $R<R_0/8$:

(i)
\begin{eqnarray*}
\fint_B \exp \lf\{\frac{c|Du(x_0)|}
{(1+ \sqrt {c_{\kz}(R^2)}R)\|g\|_{L^Q(B/4)}}\r\}^{\frac{Q}{Q-1}}\,d\mu(x_0)&&\le C;
\end{eqnarray*}

(ii) for $p\in (\frac Q2,Q)\cap (1,Q)$,
$$\lf(\fint_B |Du|^{p^\ast}\,d\mu\r)^{1/p^\ast}\le
C(1+\sqrt{c_{\kz}(R^2)}R)R\lf(\fint_{B/4}|g|^p\,d\mu\r)^{1/p}.$$
\end{thm}

Using our assumption that the support of $g$ lies in $B/4$, we deduce
following estimate on $|Du(x_0)|$ for $x_0\in B\setminus \frac 38B$.
\begin{lem}\label{l4.1}
For $p\in (\frac Q2,Q]\cap (1,Q]$, we have
$$\||Du|\|_{L^{\fz}(B\setminus \frac{3}{8}B)}\le C(1+ \sqrt {c_{\kz}(R^2)}R)R^{1-Q/p}\|g\|_{L^{p}(B/4)}.$$
\end{lem}
\begin{proof} By Theorem \ref{t3.1}, we have that for almost every $x_0\in B$,
\begin{eqnarray*}
|Du(x_0)|^2&&\le C(1+ c_{\kz}(R^2)R^2)C(u,g)^2+
\int_0^{R^2}\frac{1}{t} \int_X \lf|w_{x_0}(t,x)\psi(x)g(x)\r|p(t,x_0,x)\,d\mu(x)\,dt,
\end{eqnarray*}
where $C(u,g)=R^{-Q/2-1}\|u\|_{L^2(8B)}+R^{1-Q/p}\|g\|_{L^p(B/4)}$.
By Lemma \ref{l2.4}, we have that $\|u\|_{L^\fz(8B)}\le CR^{2-Q/p}\|g\|_{L^p(B/4)}$,
and hence,
\begin{eqnarray}\label{x4.1}
|Du(x_0)|^2&&\le C(1+ c_{\kz}(R^2)R^2)[R^{1-Q/p}\|g\|_{L^p(B/4)}]^2\nonumber\\
&&\hs+\int_0^{R^2}\frac{1}{t} \int_{B/4}|w_{x_0}(t,x)g(x)|p(t,x_0,x)\,d\mu(x)\,dt.
\end{eqnarray}

For every $x_0\in B\setminus \frac{3}{8}B$, since $\supp g\subset B/4$, we have
$d(x,x_0)>R/8$ for each $x\in B/4$. Hence, by the H\"older inequality and Lemma \ref{l2.4},
we deduce that
\begin{eqnarray*}
&&\int_0^{R^2}\frac 1t\int_{B/4}|(u\psi)(x)-T_t(u\psi)(x_0)||g(x)|p(t,x_0,x)\,d\mu(x)\,dt\\
&&\hs\le C\int_0^{R^2}\frac 1t\int_{B/4}|(u\psi)(x)-T_t(u\psi)(x_0)||g(x)|
\frac{1}{t^{Q/2}}e^{-R^2/ct}\,d\mu(x)\,dt\\
&&\hs\le C\|u\|_{L^\fz(8B)}\|g\|_{L^1(B/4)} \int_0^{R^2}
\frac{1}{t^{Q/2+1}}\lf(\frac{t}{R^2}\r)^{Q/2+1}\,dt\\
&&\hs\le C[R^{1-Q/p}\|g\|_{L^p(B/4)}]^2,
\end{eqnarray*}
which together with \eqref{x4.1} proves the lemma.
\end{proof}

Recall that for $R,\az>0$,
$\Psi_{R,\az}(t)=\frac{e^{t^{\az}}-1}{R^Q}$, and its
complementary function $\Phi_{R,1/\az}(t)$,
is equivalent to $t[\log(e+R^Qt)]^{1/\az}$.
By Lemma \ref{l3.3}, our function $u$ has a representative for which
the following holds.
\begin{lem}\label{l4.2}

(i) There exists $C>0$ such that for all $x_0\in \frac 38B$
and $x\in \frac 12B$,
\begin{eqnarray*}
|u(x_0)-u(x)|\le Cd(x_0,x)\log^{1/1^\ast}\lf(\frac{eR}{d(x_0,x)}\r)
\||Du|\|_{\Psi_{R,1^\ast}(B)}.
\end{eqnarray*}

(ii) Let $p\in (\frac Q2,Q)\cap (1,Q)$.
There exists $C>0$  such that for all $x_0\in \frac 38B$ and $x\in \frac 12B$,
\begin{eqnarray*}
|u(x_0)-u(x)|\le Cd(x_0,x)^{2-Q/p}\||Du|\|_{L^{p^\ast}(B)}.
\end{eqnarray*}
\end{lem}
\begin{proof}  Notice that by Lemma \ref{l3.3}, we have
$\||Du|\|_{L^\fz(B)}<\fz.$ Thus we may assume that $u$ is (Lipschitz) continuous in $B$.

For all $x_0\in \frac{3}{8}B$
and $x\in B/2$, $d(x,x_0)<14R/8$. We first consider the case that $d(x,x_0)\le R/8$.
Let $B_1=B(x_0,d(x,x_0))$ and $B_{0}=B(x,2d(x,x_0))$. For $j\ge 2$ and $i\ge 1$
set $B_j=2^{-1}B_{j-1}$ and $B_{-i}=2^{-1}B_{-i+1}$ inductively.
Further,
$$|u(x)-u(x_0)|\le \sum_{j=-\fz}^{\fz}|u_{B_j}-u_{B_{j+1}}|,$$
where for each $j\ge 0$, the Poincar\'e inequality yields that
\begin{eqnarray*}
|u_{B_j}-u_{B_{j+1}}|&&\le C\mathrm{diam}(B_j)
\lf(\frac{1}{\mu(B_j)}\int_{B_j} |Du|^2\,d\mu\r)^{1/2}.
\end{eqnarray*}
Applying the Orlicz-H\"older inequality \eqref{x2.4}, we have
\begin{eqnarray*}
\int_{B_j} |Du|^2\,d\mu&&\le C\||Du|^2\|_{\Psi_{R,1^\ast/2}(B)}
\|\chi_{B_j}\|_{\Phi_{R,2/1^\ast}(X)}=C\||Du|\|^2_{\Psi_{R,1^\ast}(B)}
\|\chi_{B_j}\|_{\Phi_{R,2/1^\ast}(X)},
\end{eqnarray*}
where
\begin{eqnarray*}
\|\chi_{B_j}\|_{\Phi_{R,2/1^\ast}(X)}&&=\inf\lf\{\lz>0:\,
\int_{B_j}\frac{1}{\lz}\log^{2/1^\ast}\lf(e+\frac{R^Q}{\lz}\r)\,d\mu\le 1\r\}\\
&&=\inf\lf\{\lz>0:\,
\frac{1}{\lz}\log^{2/1^\ast}\lf(e+\frac{R^Q}{\lz}\r)\le \mu(B_j)^{-1}\r\}\\
&&\le C (2^{-j}d(x_0,x))^Q \log^{2/1^\ast}\lf(\frac{eR}{2^{-j}d(x_0,x)}\r).
\end{eqnarray*}
Hence, we obtain that
\begin{eqnarray}\label{x4.2}
|u_{B_j}-u_{B_{j+1}}|&&\le C2^{-j}d(x_0,x)
\log^{1/1^\ast}\lf(\frac{eR}{2^{-j}d(x_0,x)}\r)\||Du|\|_{\Psi_{R,1^\ast}(B)}.
\end{eqnarray}

Similarly, for each $j<0$,
\begin{eqnarray*}
|u_{B_j}-u_{B_{j+1}}|&&\le C2^{j}d(x_0,x)
\log^{1/1^\ast}\lf(\frac{eR}{2^{j}d(x_0,x)}\r)\||Du|\|_{\Psi_{R,1^\ast}(B)}.
\end{eqnarray*}
Hence, for all $x_0\in \frac 38B$ and $x\in B/2$ with $d(x,x_0)\le R/8$, we obtain
\begin{eqnarray*}
|u(x_0)-u(x)|&&\le \sum_{j=-\fz}^{\fz}|u_{B_j}-u_{B_{j+1}}|\le Cd(x_0,x)
\log^{1/1^\ast}\lf(\frac{eR}{d(x_0,x)}\r)\||Du|\|_{\Psi_{R,1^\ast}(B)}.
\end{eqnarray*}

For all $x_0\in \frac 38B$ and $x\in B/2$ with $d(x,x_0)\ge R/8$,
by applying a similar approach as in the case $d(x,x_0)\le R/8$ to
the pairs $(x,y_0)$ and $(x_0,y_0)$, respectively, we obtain
\begin{eqnarray*}
|u(x_0)-u(x)|&&\le |u(x)-u(y_0)|+|u(x_0)-u(y_0)|\\
&&\le Cd(x_0,x)
\log^{1/1^\ast}\lf(\frac{eR}{d(x_0,x)}\r)\||Du|\|_{\Psi_{R,1^\ast}(B)}
\end{eqnarray*}
for all $x_0\in \frac 38B$ and $x\in B/2$, proving (i).

By the fact that $p^\ast>2$ for $p\in (\frac Q2,Q)\cap (1,Q)$ and
the H\"older inequality, we have
\begin{eqnarray*}
|u_{B_j}-u_{B_{j+1}}|&&\le C\mathrm{diam}(B_j)
\lf(\fint_{B_j} |Du|^{2}\,d\mu\r)^{1/2}\le C\mathrm{diam}(B_j)
\lf(\fint_{B_j} |Du|^{p^\ast}\,d\mu\r)^{1/p^\ast}.
\end{eqnarray*}
Using this inequality instead of \eqref{x4.2} in the ``telescope" approach above, we see that (ii)
holds, proving the lemma.
\end{proof}

\begin{prop}\label{p4.1}
(i) For $p=Q>1$, there exists $C>0$
such that for almost every $x_0\in B$,
\begin{eqnarray*}
|Du(x_0)|^2&&\le C(1+ c_{\kz}(R^2)R^2)\|g\|_{L^Q(B/4)}^2\\
&&\hs+ C\lf[\|g\|_{L^{Q}(B/4)}+\||Du|\|_{\Psi_{R,1^\ast}(B)}\r]\int_{B/4}
\frac{\log^{1/1^\ast}\lf(\frac{eR}{d(x_0,x)}\r)|g(x)|}{d(x,x_0)^{Q-1}}\,d\mu(x).
\end{eqnarray*}

(ii) For $p\in(\frac Q2,Q)\cap (1,Q)$, there exists $C>0$
such that for almost every $x_0\in B$,
\begin{eqnarray*}
|Du(x_0)|^2&&\le C(1+ c_{\kz}(R^2)R^2)[R^{1-Q/p}\|g\|_{L^p(B/4)}]^2\\
&&\hs+ C\lf[\|g\|_{L^{p}(B/4)}+\||Du|\|_{L^{p^\ast}(B)}\r]\int_{B/4}
\frac{|g(x)|}{d(x,x_0)^{Q-2+Q/p}}\,d\mu(x).
\end{eqnarray*}
\end{prop}
\begin{proof}
By \eqref{x4.1}, we have
that for almost every $x_0\in B$ and $p\in (\frac Q2,Q]\cap(1,Q]$,
\begin{eqnarray*}
\quad|Du(x_0)|^2&&\le C(1+ c_{\kz}(R^2)R^2)[R^{1-Q/p}\|g\|_{L^p(B/4)}]^2+
\int_0^{R^2}\frac{1}{t} \int_{B/4}|w_{x_0}(t,x)g(x)|p(t,x_0,x)\,d\mu(x)\,dt,
\end{eqnarray*}
where $w_{x_0}(t,x)=(u\psi)(x)-T_t(u\psi)(x_0)$

Let us first prove (i).
By Lemma \ref{l4.1}, we have that $\||Du|\|_{L^\fz(B\setminus \frac{3}8B)}\le C\|g\|_{L^Q(8B)}$.
Thus, assume $x_0\in \frac{3}{8}B$.

Now by the fact $T_t1=1$, we write
\begin{eqnarray*}
\int_{B/4} \lf|w_{x_0}(t,x)g(x)\r|p(t,x_0,x)\,d\mu(x)&&\le \int_{B/4} |u\psi(x)-u\psi(x_0)||g(x)|p(t,x_0,x)\,d\mu(x)\\
&&\hs\hs+\int_{B/4} |T_t(u\psi(x_0)-u\psi)(x_0)||g(x)|p(t,x_0,x)\,d\mu(x)\\
&&=:\mathrm{H}_{1}+\mathrm{H}_{2}.
\end{eqnarray*}

By Lemma \ref{l4.2} (i), we have
\begin{eqnarray*}
\mathrm{H}_{1}&&\le \int_{B/4} |u(x)-u(x_0)||g(x)|p(t,x_0,x)\,d\mu(x)\\
&&\le\int_{B/4} Cd(x_0,x)
\log^{1/1^\ast}\lf(\frac{eR}{d(x_0,x)}\r)\||Du|\|_{\Psi_{R,1^\ast}(B)}
|g(x)|p(t,x_0,x)\,d\mu(x)\\
&& \le C\||Du|\|_{\Psi_{R,1^\ast}(B)}\frac{\log^{1/1^\ast}\lf(\frac{eR^2}{t}\r)}{t^{(Q-1)/2}}
\int_{B/4}|g(x)|e^{-\frac{d(x,x_0)^2}{2C_1t}}\,d\mu(x).
\end{eqnarray*}
Notice that for $x\notin B/2$ and $x_0\in 3B/8$, we have $d(x,x_0)>R/8$.
For the term $\mathrm{H}_{2}$, by Lemma \ref{l4.2}(i) again, we have
\begin{eqnarray*}
&&|T_t(u\psi(x_0)-u\psi)(x_0)|\\
&&\hs\le \int_{X\setminus B/2} |u\psi(x)-u\psi(x_0)|p(t,x_0,x)\,d\mu(x)
+\int_{B/2}|u\psi(x)-u\psi(x_0)|p(t,x_0,x)\,d\mu(x)\\
&&\hs\le C\|u\|_{L^\fz(8 B)}e^{-R^2/ct}
\int_{X}p(lt,x_0,x)\,d\mu(x)\\
&&\hs\hs+\int_{B/2} Cd(x_0,x)
\log^{1/1^\ast}\lf(\frac{eR}{d(x_0,x)}\r)\||Du|\|_{\Psi_{R,1^\ast}(B)}p(t,x_0,x)\,d\mu(x)\\
&&\hs\le CR\|g\|_{L^{Q}(B/4)}\frac{t^{1/2}}{R}+
Ct^{1/2}\log^{1/1^\ast}\lf(\frac{eR^2}{t}\r)\||Du|\|_{\Psi_{R,1^\ast}(B)}\int_{X}p(lt,x_0,x)\,d\mu(x)\\
&&\hs\le C\lf[\|g\|_{L^{Q}(B/4)}+\||Du|\|_{\Psi_{R,1^\ast}(B)}\r]t^{1/2}
\log^{1/1^\ast}\lf(\frac{eR^2}{t}\r),
\end{eqnarray*}
where $l=\frac{C_2}{2C_1}$. By this estimate, we further obtain
\begin{eqnarray*}
\mathrm{H}_{2}&&\le C\lf[\|g\|_{L^{Q}(B/4)}+\||Du|\|_{\Psi_{R,1^\ast}(B)}\r]t^{1/2}
\log^{1/1^\ast}\lf(\frac{eR^2}{t}\r)
\int_{B/4} |g(x)|p(t,x_0,x)\,d\mu(x)\\
&& \le C\lf[\|g\|_{L^{Q}(B/4)}+\||Du|\|_{\Psi_{R,1^\ast}(B)}\r]
\frac{\log^{1/1^\ast}\lf(\frac{eR^2}{t}\r)}{t^{(Q-1)/2}}
\int_{B/4}|g(x)|e^{-\frac{d(x,x_0)^2}{2C_1t}}\,d\mu(x).
\end{eqnarray*}

Combining the estimates for $\mathrm{H}_{1}$ and $\mathrm{H}_{2}$, we conclude that
\begin{eqnarray*}
&&\int_{B/4} \lf|(u\psi)(x)-T_t(u\psi)(x_0)g(x)\r|p(t,x_0,x)\,d\mu(x)\\
&&\hs\le C\lf[\|g\|_{L^{Q}(B/4)}+\||Du|\|_{\Psi_{R,1^\ast}(B)}\r]
\frac{\log^{1/1^\ast}\lf(\frac{eR^2}{t}\r)}{t^{(Q-1)/2}}
\int_{B/4}|g(x)|e^{-\frac{d(x,x_0)^2}{2C_1t}}\,d\mu(x),
\end{eqnarray*}
and hence,
\begin{eqnarray*}
&&\int_0^{R^2}\frac{1}{t}\int_{B/4} \lf|(u\psi)(x)-T_t(u\psi)(x_0)g(x)\r|p(t,x_0,x)\,d\mu(x)\,dt\\
&&\hs\le C\lf[\|g\|_{L^{Q}(B/4)}+\||Du|\|_{\Psi_{R,1^\ast}(B)}\r]
\int_0^{R^2}\frac{\log^{1/1^\ast}\lf(\frac{eR^2}{t}\r)}{t^{(Q+1)/2}}\int_{B/4} |g(x)|e^{-\frac{d(x,x_0)^2}{2C_1t}}\,d\mu(x)\,dt\nonumber\\
&&\hs \le C\lf[\|g\|_{L^{Q}(B/4)}+\||Du|\|_{\Psi_{R,1^\ast}(B)}\r]\\
&&\hs\hs\times\liminf_{\dz\to 0^+}\int_{ B/4\setminus B(x_0,\dz)}\int_{\frac{d(x,x_0)^2}{R^2}}^\fz
|g(x)|\log^{1/1^\ast}\lf(\frac{eR^2s}{d(x_0,x)^2}\r)\lf(\frac{s}{d(x,x_0)^{2}}\r)^{\frac{Q-1}{2}}e^{-s}\frac{\,ds}{ s}\,d\mu(x)\nonumber\\
&&\hs\le C\lf[\|g\|_{L^{Q}(B/4)}+\||Du|\|_{\Psi_{R,1^\ast}(B)}\r]\int_{B/4}
\frac{\log^{1/1^\ast}\lf(\frac{eR}{d(x_0,x)}\r)|g(x)|}{d(x,x_0)^{Q-1}}\,d\mu(x).
\end{eqnarray*}
The desired estimate follows.

Using Lemma \ref{l4.2} (ii) instead of Lemma \ref{l4.2} (i) in the argument above, we
see that (ii) holds as well, proving the proposition.
\end{proof}

Now the main problem is reduced to estimating the Riesz potentials in Proposition \ref{p4.1}.
To this end, we establish the following boundedness of Riesz potentials.

Let $\az\in (0,Q)$ and  $\bz\in [0,\fz)$.
For a non-negative measurable function $f$ on $B_{R}(y_0)$ and $x\in B_{R}(y_0)$,
define its Riesz potential $\car_{\az,\bz}f$ by
\begin{equation*}
\car_{\az,\bz}f(x)=\int_{B_{R}(y_0)}
\frac{(\log \frac {eR}{d(x,y)})^\bz }{d(x,y)^{Q-\az}}f(y)\,d\mu(y).
\end{equation*}
It is easy to see that Riesz potential $\car_{\az,\bz}f$
is well defined for $f\in L^\fz(B)$.
Recall that $\cm$ denotes the Hardy-Littlewood maximal function
on $X$.

\begin{thm}\label{t4.2}
Let $Q\in (1,\fz)$, $\az\in (0,Q)$ and  $\bz\in [0,\fz)$. Then there exist $c,C>0$
such that for every $B_0=B_R(y_0)\subset X$ with $R< R_0$:

(i) for $p=Q/\az$,
\begin{eqnarray*}
\fint_{B_0}\exp\lf\{\frac{c\car_{\az,\bz}(|f|)}
{\|f\|_{L^{Q/\az}(B_0)}}\r\}^{\frac{Q}{Q(\bz+1)-\az}}\,d\mu\le C;
\end{eqnarray*}

(ii) for $\bz=0$ and $p\in (1,Q/\az)$,
\begin{eqnarray*}
\|\car_{\az,0}(f)\|_{L^{\frac{Qp}{Q-\az p}}(B_0)}\le C\|f\|_{L^p(B_0)}.
\end{eqnarray*}
\end{thm}
\begin{proof} Let us prove (i). Let $\phi(r)=r^{\az-Q}(\log \frac {eR}{r})^\bz$.
For $r\in (0,2R)$, write
\begin{eqnarray*}
\car_{\az,\bz}f(x)&&=\int_{B_0\cap B_r(x)}
\phi(d(x,y))f(y)\,d\mu(y)+\int_{B_0\setminus B_r(x)}
\phi(d(x,y))f(y)\,d\mu(y).
\end{eqnarray*}
In what follows, for a ball $B=B_\ro(z)$ and $k\in\zz$,
let $U_k(B):=B_{2^k\ro}(z)\setminus B_{2^{k-1}\ro}(z)$.

If $\az\in (0,Q)$, then
\begin{eqnarray*}
\int_{B_0\cap B_r(x)}
\phi(d(x,y))f(y)\,d\mu(y)&&\le \sum_{k\le 0}\int_{U_k(B_r(x))}
\phi(d(x,y))f(y)\,d\mu(y)\nonumber\\
&&\le \sum_{k\le 0}(2^kr)^{\az-Q}\lf(\log \frac {eR}{2^kr}\r)^\bz
\int_{U_k(B_r(x))}f(y)\,d\mu(y)\nonumber\\
&&\le C\sum_{k\le 0}(2^kr)^{\az}\lf(|k|\log \frac {eR}{r}\r)^\bz
\fint_{B_{2^kr}(x)}f(y)\,d\mu(y)\nonumber\\
&&\le Cr^\az\lf(\log \frac {eR}{r}\r)^\bz \cm(f)(x).
\end{eqnarray*}
On the other hand, by the H\"older inequality, we obtain
\begin{eqnarray*}
&&\int_{B_0\setminus B_r(x)}\phi(d(x,y))f(y)\,d\mu(y)\nonumber\\&&
\le \|f\|_{L^{Q/\az}(B_0)}\lf\{\int_{B_0\setminus B_r(x)}
{d(x,y)^{-Q}}\lf(\log \frac {eR}{d(x,y)}\r)^{\frac{\bz Q}{Q-\az}}
\,d\mu(y)\r\}^{\frac{Q-\az}{Q}}\nonumber\\
&&\le \|f\|_{L^{Q/\az}(B_0)}\lf\{\sum_{1\le k\le 2\log_2 R/r}\int_{U_k( B_r(x))}
{(2^kr)^{-Q}}\lf(\log \frac {eR}{2^kr}\r)^{\frac{\bz Q}{Q-\az}}
\,d\mu(y)\r\}^{\frac{Q-\az}{Q}}\nonumber\\
&&\le C\|f\|_{L^{Q/\az}(B_0)}\lf(\log \frac {eR}{r}\r)^{\bz+\frac{Q-\az}{ Q}}.
\end{eqnarray*}
By letting $r^\az=\min\{R^\az,\frac{\|f\|_{L^{Q/\az}(B_0)}}{\cm(f)(x)}\}$, we obtain
that
\begin{eqnarray*}
\car_{\az,\bz}f(x)&&\le C\|f\|_{L^{Q/\az}(B_0)}
\max\lf\{1,\lf(\log \frac {eR^\az \cm(f)(x)}{\|f\|_{L^{Q/\az}(B_0)}}\r)^{\bz+\frac{Q-\az}{ Q}}\r\}.
\end{eqnarray*}
Hence, by the H\"older inequality, we obtain
\begin{eqnarray*}
\fint_B\exp\lf\{\frac{c\car_{\az,\bz}(|f|)}
{\|f\|_{L^{Q/\az}(B_0)}}\r\}^{\frac{Q}{Q(\bz+1)-\az}}\,d\mu&&\le
C\fint_B\frac {eR^\az \cm(f)(x)}{\|f\|_{L^{Q/\az}(B_0)}}\,d\mu\\
&&\le \frac{C}{\mu(B)\|f\|_{L^{Q/\az}(B_0)}}R^\az\mu(B)^{\frac{Q-\az}{Q}}
\|\cm(f)\|_{L^{Q/\az}(B)}\le C,
\end{eqnarray*}
proving (i).

The case (ii) follows similarly, the theorem is proved.%Now let $\bz=0$. From the above estimates, we deduce that
%\begin{eqnarray*}
%\int_{B_0\cap B_r(x)}\phi(d(x,y))f(y)\,d\mu(y)&&
%\le Cr^\az \cm(f)(x).
%\end{eqnarray*}
%For all $p\in (1,Q/\az)$, applying the H\"older inequality, we obtain
%\begin{eqnarray*}
%&&\int_{B_0\setminus B_r(x)}\phi(d(x,y))f(y)\,d\mu(y)\nonumber\\&&
%\le \|f\|_{L^{p}(B_0)}\lf\{\int_{B_0\setminus B_r(x)}
%{d(x,y)^{-(Q-\az)\frac{p}{p-1}}}\,d\mu(y)\r\}^{\frac{p-1}{p}}\nonumber\\
%&&\le \|f\|_{L^{p}(B_0)}\lf\{\sum_{1\le k\le 2\log_2 R/r}\int_{U_k( B_r(x))}
%{(2^kr)^{-(Q-\az)\frac{p}{p-1}}}\,d\mu(y)\r\}^{\frac{p-1}{p}}\nonumber\\
%&&\le Cr^{\az-Q/p}\|f\|_{L^{p}(B_0)}.
%\end{eqnarray*}
%By letting $r^{Q/p}=\min\{R^{Q/p},\frac{\|f\|_{L^{p}(B_0)}}{\cm(f)(x)}\}$, we obtain
%that
%\begin{eqnarray*}
%\car_{\az,0}f(x)&&\le C[\cm(f)(x)]^{1-p\az/Q}\|f\|_{L^{p}(B_0)}^{p\az/Q},
%\end{eqnarray*}
%which implies that for $p\in (1,Q/\az)$,
%\begin{eqnarray*}
%\int_{B_0}[\car_{\az,0}f(x)]^{\frac{Qp}{Q-\az p}}\,d\mu(x)&&
%\le C\|f\|_{L^{p}(B_0)}^{\frac {p^2\az}{Q-\az p}}\int_{B_0}[\cm(f)(x)]^{p}\,d\mu(x)
%\le C\|f\|_{L^{p}(B_0)}^{\frac {Qp}{Q-\az p}},
%\end{eqnarray*}
\end{proof}

As an application of the mapping properties of the Riesz potential, we obtain the
main result of this section.
\begin{proof}[{\bf Proof of Theorem  \ref{t4.1}}]
By Proposition \ref{p4.1}, we have that for almost every $x_0\in B$,
\begin{eqnarray*}
|Du(x_0)|^2&&\le C(1+ c_{\kz}(R^2)R^2)\|g\|_{L^Q(B/4)}^2\\
&&\hs+ C\lf[\|g\|_{L^{Q}(B/4)}+\||Du|\|_{\Psi_{R,1^\ast}(B)}\r]\int_{B/4}
\frac{\log^{1/1^\ast}\lf(\frac{eR}{d(x_0,x)}\r)|g(x)|}{d(x,x_0)^{Q-1}}\,d\mu(x).
\end{eqnarray*}
Recall that for $R,\gz>0$, $\Psi_{R,\gz}(t)=\frac{1}{R^Q}(e^{t^\gz}-1)$.
  By Theorem \ref{t4.2} with $\az=1$ and $\bz=1/1^\ast$, we see that
$$G(x_0):=\int_{B/4}
\frac{\log^{1/1^\ast}\lf(\frac{eR}{d(x_0,x)}\r)|g(x)|}{d(x,x_0)^{Q-1}}\,d\mu(x)
\in \Psi_{R,1^\ast/2}(B),$$
with
\begin{eqnarray*}
\fint_B\bigg[\exp\lf\{\frac{G(x_0)}
{C\|g\|_{L^{Q}(B/4)}}\r\}^{\frac{Q}{2(Q-1)}}-1\bigg]\,d\mu(x_0)\le 1.
\end{eqnarray*}
Thus, we deduce that
\begin{eqnarray*}
&&\fint_B\bigg[\exp\bigg\{\frac{|Du(x_0)|^2}
{ C(1+ c_{\kz}(R^2)R^2)\|g\|_{L^Q(B/4)}^2+C\|g\|_{L^Q(B/4)}\||Du|\|_{\Psi_{R,1^\ast}(B)}}
\bigg\}^{\frac{Q}{2(Q-1)}}-1\bigg]\,d\mu(x_0)\\
&&\hs\le \fint_B\bigg[\exp\bigg\{1+\frac{G(x_0)}{C\|g\|_{L^Q(B/4)}}
\bigg\}^{\frac{Q}{2(Q-1)}}-1\bigg]\,d\mu(x_0)\le 1,
\end{eqnarray*}
which implies that
\begin{eqnarray*}
\||Du|\|_{\Psi_{R,1^\ast}(B)}^2&&\le C(1+ c_{\kz}(R^2)R^2)\|g\|_{L^Q(B/4)}^2+
C\|g\|_{L^Q(B/4)}\||Du|\|_{\Psi_{R,1^\ast}(B)}\\
&&\le C(1+ c_{\kz}(R^2)R^2)\|g\|_{L^Q(B/4)}^2+\frac{1}{2}\||Du|\|^2_{\Psi_{R,1^\ast}(B)},
\end{eqnarray*}
and hence,
\begin{eqnarray*}
\fint_B \exp \bigg\{\frac{|Du(x_0)|}
{c(1+ \sqrt {c_{\kz}(R^2)}R)\|g\|_{L^Q(B/4)}}\bigg\}^{\frac{Q}{Q-1}}\,d\mu(x_0)&&\le C,
\end{eqnarray*}
proving (i).

Now for $p\in (\frac Q2,Q)\cap (1,Q)$, by Proposition \ref{p4.1},
we have that for almost every $x_0\in B$,
\begin{eqnarray*}
|Du(x_0)|^2&&\le C(1+ c_{\kz}(R^2)R^2)[R^{1-Q/p}\|g\|_{L^p(B/4)}]^2\\
&&\hs+ C\lf[\|g\|_{L^{p}(B/4)}+\||Du|\|_{L^{p^\ast}(B)}\r]\int_{B/4}
\frac{|g(x)|}{d(x,x_0)^{Q-2+Q/p}}\,d\mu(x).
\end{eqnarray*}
According to Theorem \ref{t4.2} (ii), we have that
$$\wz G(x_0):=\int_{B/4}
\frac{|g(x)|}{d(x,x_0)^{Q-2+Q/p}}\,d\mu(x)
\in L^{\frac{Qp}{2(Q-p)}}(B),$$
which implies that
\begin{eqnarray*}
\||Du|^2\|_{L^{\frac{Qp}{2(Q-p)}}(B)}&&\le C(1+ c_{\kz}(R^2)R^2)[R^{1-Q/p}\|g\|_{L^p(B/4)}]^2\mu(B)^{\frac{2(Q-p)}{Qp}}\\
&&\hs+ C\lf[\|g\|_{L^{p}(B/4)}+\||Du|\|_{L^{p^\ast}(B)}\r]\|\wz G\|_{L^{\frac{Qp}{2(Q-p)}}(B)}\\
&&\le C(1+ c_{\kz}(R^2)R^2)\|g\|_{L^p(B/4)}^2+C\lf[\|g\|_{L^{p}(B/4)}+\||Du|\|_{L^{p^\ast}(B)}\r]\|g\|_{L^p(8B)}\\
&&\le C(1+ c_{\kz}(R^2)R^2)\|g\|_{L^p(B/4)}^2+\frac 12\||Du|\|_{L^{p^\ast}(B)}^2.
\end{eqnarray*}
Thus, we obtain that $\||Du|\|_{L^{p^\ast}(B)}\le C(1+ \sqrt {c_{\kz}(R^2)}R)\|g\|_{L^p(B/4)}$,
proving the theorem.
\end{proof}

\section{Proofs of the main results}
\hskip\parindent In this section, we prove the main results of this paper.
By Theorem \ref{t4.1}, our proofs of Theorem \ref{t1.1} and Theorem \ref{t1.2}
are reduced to approximation arguments and use of Cheeger-harmonic functions.

We first prove Theorem \ref{t1.1}.
\begin{proof}[{\bf Proof of Theorem \ref{t1.1}}]
For each $k\in \cn$, let $g_k=g\chi_{8B\cap \{|g|\le k \}}$.
Then, by Lemma \ref{l2.6}, there exist $u_k\in H^{1,2}_0(256B)$ such that
$\Delta u_k=g_k$ in $256B$. By Theorem \ref{t4.1},
we obtain
\begin{eqnarray*}
\fint_{32B} \exp \bigg\{\frac{|D u_k(x_0)|}
{c(1+ \sqrt {c_{\kz}(R^2)}R)\|g_k\|_{L^Q(8B)}}\bigg\}^{\frac{Q}{Q-1}}\,d\mu(x_0)&&\le C.
\end{eqnarray*}
Moreover, By Lemma \ref{l2.5} and the Sobolev inequality, we have
\begin{eqnarray*}
\|u_k-u_j\|_{L^2(256 B)}+\||D(u_k-u_j)|\|_{L^2(256B)}&&\le C_R\|g_k-g_j\|_{L^Q(8 B)}\to 0,
\end{eqnarray*}
as $k,j\to\fz$. Hence $\{u_k\}_{k\in\cn}$ is a Cauchy sequence in $H^{1,2}_0(256 B)$,
and there exists $\wz u\in H^{1,2}_0(256B)$ such that $\lim_{k\to\fz}u_k=\wz u$ in
$H^{1,2}_0(256B)$ and $\Delta \wz u=g\chi_{8B}$ in $256B$.
By Theorem \ref{t4.1} (i) again, we further deduce that
\begin{equation*}
\||Du_k-Du_j|\|_{\Psi_{32R,1^\ast}(32B)}\le
C(1+ \sqrt {c_{\kz}(R^2)}R)\|g_k-g_j\|_{L^{Q}(8B)}\to 0
\end{equation*}
as $k,j\to\fz$, which implies that
\begin{equation}\label{x5.1}
\||D\wz u|\|_{\Psi_{32R,1^\ast}(32B)}\le
C(1+ \sqrt {c_{\kz}(R^2)}R)\|g\|_{L^{Q}(8B)}.
\end{equation}

On the other hand, since
\begin{equation*}
\int_{8B} D\wz u(x)\cdot D\phi(x)\,d\mu(x)=-\int_{8B}
 g(x)\phi(x)\,d\mu(x)=\int_{8 B}Du(x)\cdot D\phi(x)\,d\mu(x),
\ \ \forall \phi\in H_0^{1,2}(8 B),
\end{equation*}
we see that $u-\wz u$ is Cheeger-harmonic in $8B$.
By \cite{krs} or Theorem \ref{t3.1} with $g=0$, we have
\begin{eqnarray*}
\||D(u-\wz u)|\|_{L^\fz(B)}&&\le C(1+ \sqrt {c_{\kz}(R^2)}R)\frac{\|u-\wz u\|_{L^2(8 B)}}{R^{Q/2+1}}
\le C(1+ \sqrt {c_{\kz}(R^2)}R)\lf(\frac{\|u\|_{L^2(8B)}}{R^{Q/2+1}}+\|g\|_{L^{Q}(8B)}\r),
\end{eqnarray*}
which together with \eqref{x5.1} implies that
\begin{eqnarray*}
\fint_{B} \exp \bigg\{\frac{|D u(x_0)|}
{c(1+ \sqrt {c_{\kz}(R^2)}R)C(u,g)}\bigg\}^{\frac{Q}{Q-1}}\,d\mu(x_0)&&\le C,
\end{eqnarray*}
where $C(u,g)=\frac{\|u\|_{L^2(8B)}}{R^{Q/2+1}}+\|g\|_{L^{Q}(8B)}$, completing the proof
of Theorem \ref{t1.1}.
\end{proof}

Observe that in Theorem \ref{t4.1}, the range of $p$ lies in $(\frac Q2,Q)\cap (1,Q)$.
Thus, to obtain the results for all $p\in (2_\ast,Q)\cap (1,Q)$, we need some extra
estimates. Notice that $(\frac Q2,Q)\cap (1,Q) \neq (2_\ast,Q)\cap (1,Q)$ only for $Q>2$.

We want to
use the interpolation theory to study the case of $p\in (2_\ast,\frac Q2]$ when $Q>2$.
To this end, let us recall the
nonincreasing rearrangement function.
For a measurable function $f$, let $\sz_f$ denote its distribution function; then
its nonincreasing rearrangement function, $f^\ast$, is defined by letting for all $t>0$,
$f^\ast(t)=\inf\{s:\, \sz_f(s)\le t\}.$

We also need the following Hardy's inequalities; see \cite[p.196]{sw}.
\begin{lem}\label{l5.1}
Let $q\ge 1$ $r>0$ and $g$ be a nonnegative function defined on $(0,\fz)$.
Then

(i) $(\int_0^\fz[\int_0^tg(u)\,du]^qt^{-r-1}\,dt)^{1/q}\le
(q/r)(\int_0^\fz[ug(u)]^qu^{-r-1}\,du)^{1/q};$

(ii) $(\int_0^\fz[\int_t^\fz g(u)\,du]^qt^{r-1}\,dt)^{1/q}\le
(q/r)(\int_0^\fz[ug(u)]^qu^{r-1}\,du)^{1/q}.$
\end{lem}

\begin{prop}\label{p5.1}
Let $Q>2$ and $p\in (2_\ast,\frac Q2]$. Suppose that
$u\in H^{1,2}_0(256B)$, $g\in L^\fz(X)$ with $\supp g\subset 8B$,
and $\Delta u=g$ in $256B$,
where $B=B_R(y_0)$ with $256B\subset\subset \Omega$.
Then $|Du|\in L^{p^\ast}(32B)$ with
$$\||Du|\|_{L^{p^\ast}(32B)}\le C(1+\sqrt{c_{\kz}(R^2)}R)\|g\|_{L^p(8B)}.$$
\end{prop}
\begin{proof} For $t>0$, define
\begin{equation*}
\begin{array}[C]{l}
g^t(x):={\lf\{
\begin{array}{ll}
g(x)\quad\quad &\quad\mbox{if} \quad |g(x)|>g^\ast(t);\\
0 &\quad{\rm if} \quad  |g(x)|\le g^\ast (t)
\end{array}
\r.}
\end{array}
\end{equation*}
and $g_t:=g-g^t$. We then have
\begin{equation*}
\begin{array}[C]{l}
(g^t)^\ast(s)\le {\lf\{
\begin{array}{ll}
g^\ast(s)\quad\quad &\quad\mbox{if} \quad s\in (0,t);\\
 0 &\quad{\rm if} \quad  s\ge t
\end{array}
\r.}
\end{array}\quad\mathrm{ and }
\end{equation*}
\begin{equation*}
\begin{array}[C]{l}
(g_t)^\ast(s)\le {\lf\{
\begin{array}{ll}
g^\ast(t)\quad\quad &\quad\mbox{if} \quad s\in (0,t);\\
 g^\ast(s) &\quad{\rm if} \quad  s\ge t.
\end{array}
\r.}
\end{array}\quad{\null}
\end{equation*}
Notice here that, for $t\ge \mu(8B)$, $g^t=g$ and $g_t=0.$

Let $G$ be the Green function on $256B$ such that for each $h\in L^\fz(256B)$,
$v:=\int_{256B}Gh\,d\mu\in H^{1,2}_0(256B)$ and $\Delta v=h$ in $256B$;
see \cite{bm}.
Write
$$u=\int_{256B}Gg\,d\mu=\int_{256B}Gg^t\,d\mu+\int_{256B}Gg_t\,d\mu=:u_1+u_2.$$
Fix a $q\in (\frac Q2,Q)$. By using Theorem \ref{t4.1} (ii) and Lemma \ref{l2.5}, we obtain
\begin{eqnarray*}
&&\||Du|\|_{L^{p^\ast}(32B)}\\
&&\hs\le \||Du_1|+|Du_2|\|_{L^{p^\ast}(32B)}\\
&&\hs\le C\lf(\int_0^{\fz} \lf[|Du_1\chi_{32B}|^\ast(t)
+|Du_2\chi_{32B}|^\ast(t)\r]^{p^\ast}\,dt\r)^{1/p^\ast}\\
&&\hs\le C\lf(\int_0^{\fz} \lf[t^{-\frac 12}\|g^t\|_{L^{2_\ast}(8B)}\r]^{p^\ast}\,dt\r)^{1/p^\ast}+
C(1+\sqrt{c_{\kz}(R^2)}R)\lf(\int_0^{\fz} \lf[t^{-\frac{1}{q^\ast}}\|g_t\|_{L^q(8B)}\r]^{p^\ast}\,dt\r)^{1/p^\ast}\\
&&\hs=: \mathrm {H}_1+\mathrm {H}_2.
\end{eqnarray*}

By the assumption that $p^\ast>2$ and Hardy's inequality (Lemma \ref{l5.1}(i)), we obtain
\begin{eqnarray*}
 \mathrm {H}_1&&\le
 C\lf(\int_0^\fz t^{-\frac {p^\ast}2}\lf(\int_0^t[g^\ast(s)]^{2_\ast}
 \,ds\r)^{p^\ast/2_\ast}\,dt\r)^{1/p^\ast}\\
 &&\le C\lf(\int_0^\fz t^{-\frac {p^\ast}2}[t^{\frac{1}{2_\ast}}
 g^\ast(t)]^{p^\ast}\,dt\r)^{1/p^\ast}
 \le C\lf(\int_0^\fz [t^{\frac{1}{2_\ast}-\frac 12+\frac 1{p^\ast}}
 g^\ast(t)]^{p^\ast}\frac{\,dt}{t}\r)^{1/p^\ast}\\
 &&\le C\lf(\int_0^\fz [t^{\frac 1p}
 g^\ast(t)]^{p^\ast}\,dt\r)^{1/p^\ast}\le C\|g\|_{L^p(8B)}.
\end{eqnarray*}
Similarly, we have $\mathrm {H}_2\le C(1+\sqrt{c_{\kz}(R^2)}R)\|g\|_{L^p(8B)}$ (see \cite{sw}),
and the desired estimate follows, proving the proposition.
\end{proof}

We now are in position to prove Theorem \ref{t1.2}. The proof is similar to that
of Theorem \ref{t1.1}. We give it for completeness.
\begin{proof}[{\bf Proof of Theorem \ref{t1.2}}]
For each $k\in \cn$, let $g_k=g\chi_{8B\cap \{|g|\le k \}}$.
Then there exists $u_k\in H^{1,2}_0(256B)$ such that
$\Delta u_k=g_k$ in $256B$. By Theorem \ref{t4.1} (ii) and Proposition \ref{p5.1},
we obtain that for all $p\in (2_\ast,Q)\cap (1,Q)$,
\begin{eqnarray*}
\||D u_k|\|_{L^{p^\ast}(32B)}\le C(1+ \sqrt {c_{\kz}(R^2)}R)\|g_k\|_{L^p(8B)}.
\end{eqnarray*}
By Lemma \ref{l2.5} and the Sobolev inequality, we have
\begin{eqnarray*}
\|u_k-u_j\|_{L^2(256 B)}+\||D(u_k-u_j)|\|_{L^2(256B)}&&\le C_R\|g_k-g_j\|_{L^p(8 B)}\to 0,
\end{eqnarray*}
as $k,j\to\fz$. Hence $\{u_k\}_{k\in\cn}$ is a Cauchy sequence in $H^{1,2}_0(256 B)$,
and there exists $\wz u\in H^{1,2}_0(256B)$ such that $\lim_{k\to\fz}u_k=\wz u$ in
$H^{1,2}_0(256B)$ and $\Delta \wz u=g\chi_{8B}$ in $256B$.
By Theorem \ref{t4.1} (ii) and Proposition \ref{p5.1} again, we further deduce that
\begin{equation*}
\||Du_k-Du_j|\|_{L^{p^\ast}(32B)}\le
C(1+ \sqrt {c_{\kz}(R^2)}R)\|g_k-g_j\|_{L^{p}(8B)}\to 0
\end{equation*}
as $k,j\to\fz$, which implies that
\begin{equation}\label{5.2}
\||D\wz u|\|_{L^{p^\ast}(32B)}\le
C(1+ \sqrt {c_{\kz}(R^2)}R)\|g\|_{L^{p}(8B)}.
\end{equation}

By the fact that $\Delta \wz u=g\chi_{8B}$ in $256B$, we deduce that
\begin{equation*}
\int_{8B} D\wz u\cdot D\phi\,d\mu=-\int_{8B} g\phi\,d\mu=\int_{8B}  Du\cdot D\phi\,d\mu,
\ \ \forall \phi\in H_0^{1,2}(8B),
\end{equation*}
which implies that $u-\wz u$ is Cheeger-harmonic in $8B$.
By \cite{krs} or Theorem \ref{t3.1} with $g=0$, we have
\begin{eqnarray*}
\||D(u-\wz u)|\|_{L^\fz(B)}&&\le C(1+ \sqrt {c_{\kz}(R^2)}R)\frac{\|u-\wz u\|_{L^2(8 B)}}{R^{Q/2+1}}\\
&&\le C(1+ \sqrt {c_{\kz}(R^2)}R)\lf(\frac{\|u\|_{L^2(8B)}}{R^{Q/2+1}}+R^{1-Q/p}\|g\|_{L^{p}(8B)}\r),
\end{eqnarray*}
which together with \eqref{5.2} implies that
\begin{eqnarray*}
\lf(\fint_B|Du|^{p^\ast}\,d\mu\r)^{1/p^\ast}&&\le
C\||D(u-\wz u)|\|_{L^\fz(B)}+
C(1+ \sqrt {c_{\kz}(R^2)}R)\mu(B)^{-1/p^\ast}\|g\|_{L^{p}(8B)}\\
&&\le C(1+\sqrt{c_{\kz}(R^2)}R)\lf\{R^{-1}\lf(\fint_{8B}|u|^2\,d\mu\r)^{1/2}
+R\lf(\fint_{8B}|g|^p\,d\mu\r)^{1/p}\r\}.
\end{eqnarray*}
This completes the proof of Theorem \ref{t1.2}.
\end{proof}

At last, we use Theorem \ref{t1.2} to prove the H\"older continuity of solutions
to Poisson equations.
\begin{proof}[{\bf Proof of Corollary \ref{c1.1}}]
For almost all $x,y\in B=B_r(y_0)$ with $256B\subset\subset \Omega$,
by Theorem \ref{t1.2} and the Poincar\'e inequality, similarly to the ``telescope" approach in
Lemma \ref{l4.2}, we have that for almost all $x,y\in B$,
%\begin{eqnarray*}
%|u_{B_j}-u_{B_{j+1}}|&&\le C \mathrm{diam}(2B_j)
%\lf(\frac{1}{\mu(2B_j)}\int_{2B_j} |Du(x)|^2\,d\mu(x)\r)^{1/2}\\
%&&\le C \mathrm{diam}(2B_j)
%\lf(\frac{1}{\mu(2B_j)}\int_{2B_j} |Du(x)|^{p^\ast}\,d\mu(x)\r)^{1/p^\ast}\\
%&&\le C\mathrm{diam}(2B_j)^{2-Q/p}\||Du|\|_{L^{p^\ast}(2B)}\\
%&&\le C\mathrm{diam}(2B_j)^{2-Q/p}(1+\sqrt{c_{\kz}(R^2)}R)\lf\{R^{-1}\lf(\fint_{10B}|u|^2\,d\mu\r)^{1/2}
%+R\lf(\fint_{10B}|g|^p\,d\mu\r)^{1/p}\r\}
%\end{eqnarray*}
\begin{eqnarray*}
|u(x)-u(y)|\le Cd(x,y)^{2-Q/p}(1+\sqrt{c_{\kz}(R^2)}R)\lf\{R^{-1}\lf(\fint_{10B}|u|^2\,d\mu\r)^{1/2}
+R\lf(\fint_{10B}|g|^p\,d\mu\r)^{1/p}\r\}.
\end{eqnarray*}
From this, we conclude that $u$ can be extended to a locally H\"older
continuous function in $\Omega$, which completes
the proof of Corollary \ref{c1.1}.
\end{proof}

\subsection*{Acknowledgment}
\hskip\parindent  The author would like to thank the referees for many kind suggestions.
He is grateful to his supervisor Professor
Pekka Koskela for posing the problem and many kind suggestions,
and wishes to express deeply thanks to Professor Dachun Yang for his encouragement and suggestions,
and to Yuan Zhou for many kind suggestions.

\noindent Renjin Jiang\\
Department of Mathematics and Statistics\\
University of Jyv\"{a}skyl\"{a}\\
P.O. Box 35 (MaD)\\
FI-40014\\
Finland

\noindent{\it E-mail address}: \texttt{renjin.r.jiang@jyu.fi}
\end{document}